\documentclass{llncs}

\title{The Third Proof of Lov\'asz's Cathedral Theorem}
\author{Nanao Kita}

\institute{Keio University, Yokohama, Japan\\
           \email{kita@a2.keio.jp}}
\date{\today}

\usepackage{amsmath, amssymb}
\usepackage{graphicx}
\usepackage{cite}
\usepackage{setspace}
\usepackage{color}
\usepackage{comment}

\spnewtheorem{fact}{Fact}{\bfseries}{\rmfamily}
\spnewtheorem{cclaim}{Claim}{\bfseries}{\rmfamily}
\spnewtheorem*{acknowledgement}{Acknowledgement}{\bfseries}{\rmfamily}
\spnewtheorem*{proofof}{Proof of}{\itfamily}{\rmfamily}
\spnewtheorem*{notes}{Note}{\bfseries}{\rmfamily}

\newcommand{\matchablesp}{factorizable~}

\newcommand{\zero}{balanced~}

\newcommand{\exposed}{exposed}

\newcommand{\Vg}{V(G)}

\newcommand{\yield}{\triangleleft}
\newcommand{\gpart}[1]{\mathcal{P}(#1)}
\newcommand{\pargpart}[2]{\mathcal{P}_{#1}(#2)}
\newcommand{\up}[1]{\mathcal{U}(#1)}

\newcommand{\parup}[2]{\mathcal{U}_{#1}(#2)}
\newcommand{\parupstar}[2]{\mathcal{U}^{*}_{#1}(#2)}
\newcommand{\vup}[1]{U(#1)}
\newcommand{\vupstar}[1]{U^*(#1)}
\newcommand{\vparup}[2]{U_{#1}(#2)}
\newcommand{\vparupstar}[2]{U^*_{#1}(#2)}

\newcommand{\what}{\hat}
\renewcommand{\Gamma}{N}

\newcommand{\astar}[2]{X}

\renewcommand{\Delta}{\triangle}
\newcommand{\gsim}[1]{\sim_{#1}}

\begin{document}
\maketitle

\pagestyle{plain}

\begin{abstract}
%cathedral_abstract
%This paper is on matching theory. 
A graph $G$ with a perfect matching is called saturated 
if $G+e$ has more perfect matchings than $G$ 
for any edge $e$ that is not  in $G$. 
%if an addition of an arbitrary complement edge creates a new perfect matching. 
Lov\'asz gave a 
 characterization of the saturated graphs called 
the cathedral theorem,  with some applications 
to the enumeration problem of  perfect matchings, and 
later Szigeti gave another proof. 
In this paper, 
we give a new proof 
with our preceding works  
which revealed canonical structures of general graphs with 
perfect matchings. 
Here, the cathedral theorem 
is derived in quite a natural way, providing more refined or generalized properties. 
Moreover, 
the new proof shows that 
it can be proved without using the Gallai-Edmonds structure theorem. 
\end{abstract}

\section{Introduction}\label{sec:intro}

A graph with a perfect matching is called {\em factorizable}. 
A factorizable graph $G$,  with the edge set $E(G)$,  is called 
{\em saturated} 
if $G+e$ has more perfect matchings than $G$ 
for any edge $e\not\in E(G)$. 
%if any addition of an arbitrary complement edge 
%creates a new perfect matching. 
There is a constructive characterization of 
the saturated graphs known as the {\em cathedral theorem}%
~\cite{lovasz1972b, lp1986, szigeti1993, szigeti2001}. 
Counting the number of perfect matchings is one of the most 
fundamental enumeration problems, 
which has applications to physical science, 
and the cathedral theorem is known to be useful for such a counting problem. 
For a  given factorizable graph, 
we can obtain a saturated graph which 
possesses the same family of perfect matchings  
by adding appropriate edges repeatedly. 
Many matching-theoretic structural properties are 
preserved by this procedure. 
Therefore, we can find several 
properties on perfect matchings 
of factorizable graphs using the cathedral theorem, 
such as relationships between 
the number of perfect matchings of a given factorizable graph 
and its structural properties 
such as its connectivity~\cite{lp1986}
 or the numbers of vertices and edges~\cite{hswy2012}.

\begin{comment}
Counting the number of perfect matchings is one of the most basic
 enumeration problem,  
which has applications to physical science~\cite{lp1986}, 
and the cathedral theorem is known to be useful for  
investigating relationship between 
the number of perfect matchings of a given graph and 
its properties~\cite{lovasz1972b, lp1986, hswy2012}.
\end{comment} 
% 

The cathedral theorem was originally 
given by Lov\'asz~\cite{lovasz1972b} (see also \cite{lp1986}), 
and later another proof was given by Szigeti~\cite{szigeti1993, szigeti2001}. 
%and was later given another proof by Szigeti~\cite{szigeti1993, szigeti2001}. 
\begin{comment}
The cathedral theorem was originally 
given by Lov\'asz~\cite{lovasz1972b} (see also \cite{lp1986}), 
with some applications to 
the problem of 
estimating the number of all the perfect matchings~\cite{lovasz1972b, lp1986, hswy2012}%
---a basic enumeration problem 
with some applications to physical science~\cite{lp1986}---,  
and later given another proof by Szigeti~\cite{szigeti1993, szigeti2001}. 
\end{comment}
%
Lov\'asz's proof is  
based on the Gallai-Edmonds structure theorem~\cite{lp1986}, 
which is one of the most powerful theorem in matching theory. 
Any graph $G$ has a partition of its vertices into three parts, 
some of which might be empty, 
so-called $D(G), A(G)$, and  $C(G)$~\cite{lp1986}, 
which we call in this paper the {\em Gallai-Edmonds partition}. 
The property that $A(G)$ forms 
a barrier with certain special properties 
is called the {\em Gallai-Edmonds structure theorem}~\cite{lp1986}. 
The Gallai-Edmonds structure theorem 
tells non-trivial structures only 
for non-factorizable graphs, 
because it treats factorizable graphs as irreducible. 
Thus, Lov\'asz proved the cathedral theorem 
by applying  the Gallai-Edmonds structure theorem 
to non-factorizable subgraphs of saturated graphs. 

Szigeti's proof  is based on 
some results on the {\em optimal ear-decompositions} 
by Frank~\cite{frank1993},  
which is also based on the Gallai-Edmonds structure theorem 
and is not a ``matching-theory-closed'' notion,
while the cathedral theorem itself is closed. 

The cathedral theorem is outlined as follows: 

\begin{itemize}
\item 
There is a constructive characterization of the saturated graphs with 
an operation called the {\em cathedral construction}. 
\item 
A set of edges of a saturated graph is a perfect matching 
if and only if it is a disjoint union of 
 perfect matchings of each ``component part'' of 
the cathedral construction that creates the saturated graph. 
\item 
For each saturated graph, 
the way to construct it by the cathedral construction 
uniquely exists. 
\item 
There is a relationship between 
the cathedral construction and the Gallai-Edmonds partition. 
\end{itemize}

\begin{comment}
\begin{itemize}
\item 
a constructive characterization of the saturated graphs with 
the operation {\em cathedral construction}, 
\item 
preservation of perfect matchings 
regarding cathedral construction, 
\item the uniqueness of  cathedral construction, and  
\item 
a relationship between cathedral construction and  the Gallai-Edmonds partition. 
\end{itemize}
\end{comment}

In our preceding works~\cite{kita2012a, kita2012b, kita2012f, kita2013a}, 
we introduced canonical structure theorems which 
tells non-trivial structures for general factorizable graphs. 
Based on these results, 
we provide yet another proof of the cathedral theorem in this paper. 
The features of the new proof are the following:  
First, it is quite natural and provides new facts as by-products. 
%The first thing is it being quite natural and providing new facts as by-product.  
The notion of ``saturated'' is defined by edge-maximality. 
%The notion of ``saturated'' represents a kind of edge-maximality by its definition. 
By considering this edge-maximality over the canonical structures of factorizable graphs, 
we obtain the new proof in quite a natural way. 
%
\begin{comment}
First, 
the cathedral theorem,
 which is a structure theorem of 
 the special class of factorizable graphs
   defined by the notion that represents 
   a kind of maximality, i.e., 
   ``saturated'',  
 can be derived in quite a natural way 
 as a consequence of considering maximality over 
 canonical structures of  general factorizable graphs.  
\end{comment}
%
%
\begin{comment}
First, 
the cathedral theorem,
 which is a structure theorem of 
 the special class of factorizable graphs
   defined by a certain extremal condition i.e. 
 ``saturated'',  
 can be derived in quite a natural way 
as an extremal consequence of 
canonical structures of  general factorizable graphs.  
\end{comment}
%
Therefore, our proof reveals 
the essential structure that underlies the cathedral theorem, 
%what essentially lies  underneath the cathedral theorem,  
and provides a bit more refined or generalized statements 
from the point of view of 
the canonical structure of general factorizable graphs. 

Second, 
it shows that 
the cathedral theorem can be proved without  
 the Gallai-Edmonds structure theorem 
nor the notion of barriers, 
since 
our previous works, as well as the proofs presented in this paper, 
are obtained 
without them. 
Even the portion of the statements of the cathedral theorem 
stating its relationship to the Gallai-Edmonds partition 
can be obtained without them.

In Section~\ref{sec:pre}, 
we give notations, definitions, 
and some preliminary facts on matchings used in this paper. 
In Section~\ref{sec:outline} we present an outline of
 how we give the new proof of the cathedral theorem. 
Section~\ref{sec:canonical} 
is to present our previous works~\cite{kita2012a, kita2012b}: 
the canonical structure theorems for general factorizable graphs. 
In Section~\ref{sec:alternating}, 
we further consider the theorems in Section~\ref{sec:canonical} 
and show one of the new theorems, 
which later turns out to provide a  generalized version of the part of the cathedral theorem regarding the Gallai-Edmonds partition. 
%; 
%later in the succeeding section, 
%this theorem turns out to provide 
%a  generalized version of 
%the part of the cathedral theorem
%regarding the Gallai-Edmonds partition. 
%
\begin{comment}
In Section~\ref{sec:alternating}, 
we show one of the new theorems, 
which is a  generalized version of 
the part of the cathedral theorem
regarding the Gallai-Edmonds partition. 
\end{comment}
%
In Section~\ref{sec:proof}, 
we complete the new proof of the cathedral theorem. 
Finally, in Section~\ref{sec:conclusion}, 
we conclude this paper. 
Some basic properties on matchings, i.e., Properties \ref{prop:path2root}, \ref{prop:allowed}, \ref{prop:deletable},  \ref{prop:separating2saturated}, 
and \ref{prop:complement} are presented in Appendix.

\section{Preliminaries}\label{sec:pre}

\subsection{Notations and Definitions}

Here, we list some standard notations and definitions, 
most of which are  given by Schrijver~\cite{schrijver2003}. 
For general accounts on matchings,  see also Lov\'asz and Plummer~\cite{lp1986}. 

We define the \textit{symmetric difference} of two sets $A$ and $B$ 
%which have a common superset 
as $(A\setminus B) \cup (B\setminus A)$ and denote it by $A\Delta B$.  
For a graph $G$, we denote the vertex set of $G$ by $V(G)$ 
and the edge set by $E(G)$ and write $G = (V(G), E(G))$.  
Hereafter for a while let $G$ be a graph and let $X\subseteq V(G)$. 
The subgraph of $G$ induced by $X$ is denoted by $G[X]$, 
and $G-X$ means $G[\Vg\setminus X]$. 
We define the \textit{contraction} of $G$ by $X$ as 
the graph arising from contracting each edge of $E(G[X])$ 
into one vertex, and denote it by
$G/X$.  
We denote the subgraph of  $G$ determined by $F\subseteq E(G)$ by $G . F$. 
%$G/X := G/E(G[X]) := G/e_1,\ldots, e_k$, 
%where $E(G[X]) = \{e_1,\ldots, e_k\}$. 

Let $G$ be a subgraph of a graph $\what{G}$, 
%Let $\what{G}$ be a graph such that $G$ is a subgraph of $\what{G}$, 
and let $e = xy\in E(\what{G})$. 
If $G$ does not have an edge joining $x$ and $y$, 
then we call $xy$ a {\em complement edge} of $G$. 
The graph $G+e$ denotes the graph $(V(G)\cup \{x,y\}, E(G)\cup\{e\})$, 
and $G-e$ the graph $(V(G), E(G)\setminus \{e\})$. 
For $F = \{e_1,\ldots, e_k\} \subseteq E(\what{G})$, 
we define $G+F := G+e_1+\cdots+e_k$ and $G-F: = G-e_1-\cdots-e_k$. 

%For $X\subseteq V(G)$, 
We define the set of {\em neighbors} of $X$ as 
the vertices in $V(G)\setminus X$ that 
are joined to some vertex of $X$, 
and denote it by $\Gamma_{G}(X)$. 
Given $Y, Z\subseteq V(G)$, 
the set $E_{G}[Y, Z]$ denotes the edges joining $Y$ and $Z$, 
and $\delta_{G}(Y)$ denotes $E_{G}[Y, V(G)\setminus Y]$.

%We treats paths and circuits as graphs. 

%
A set of edges is called a {\em matching} 
if no two of them share end vertices.  
A matching of cardinality $|V(G)|/2$ (resp. $|V(G)|/2 - 1$) is called 
a {\em  perfect matching} (resp. a {\em  near-perfect matching}). 
Hereafter for a while let $M$ be a matching of a graph $G$. 
We say $M$ \textit{ exposes } a vertex $v\in V(G)$
if $\delta_{G}(v)\cap M = \emptyset $. 
\begin{comment}
For  a matching $M$ of $G$ and $u\in V(G)$,
$u'$ denote the vertex such that $uu'\in M$, if $u$ is not exposed.  
\end{comment}

In this paper, we treat paths and circuits as graphs. 
For a path or circuit $Q$ of $G$, 
$Q$ is $M$-{\em  alternating}
%Let $Q\subseteq G$ be a path or circuit. 
\begin{comment}
we call $Q$  $M$-{\em  alternating}
\end{comment}
%We call $Q$ as   $M$-{\em  alternating}
%if the edge in $M$ and $\Eg\setminus M$ appear 
%alternately on $P$.
if $E(Q)\setminus M$ is a matching of $Q$;  
in other words, if edges of $M$ and $E(Q) \setminus M$ appear alternately in $Q$.
Let $P$ be an $M$-alternating path of $G$  with end vertices $u$ and $v$.
If $P$ has an even number of edges 
and $M\cap E(P)$ is a near-perfect matching of $P$ exposing only $v$, 
we call it an $M$-{\em  \zero path} from $u$ to $v$.
We regard a trivial path, that is, a path composed of 
one vertex and no edges  as an $M$-\zero path.
If $P$ has an odd number of edges and 
$M\cap E(P)$ (resp. $E(P)\setminus M$) is a perfect matching of $P$,
we call it $M$-{\em  saturated} (resp. $M$-{\em  \exposed}). 

%

%
%Let $X\subseteq V(G)$.
A path $P$ of $G$  is an {\em  ear relative to} $X$ 
if both end vertices of $P$ are in $X$ while internal vertices are not. 
Also, a circuit $C$ is an ear relative to $X$ if 
exactly one vertex of $C$ is in $X$.
%So do we to a circuit  if exactly one vertex of it is in $X$.
For simplicity, we call the vertices of $V(P)\cap X$ 
{\em end vertices} of $P$, even if $P$ is a circuit.
For an ear $P$ of $G$ relative to $X$,
%Particularly, if  $P-V(H)$ is an $M$-saturated path,
we call it an $M$-{\em  ear}
if  $P-X$ is an $M$-saturated path. 

\if0%%%%%%%%%% moved %%%%%%%%%%%%%%%%%%%%%%%%%%%
We say $X\subseteq V(G)$ is \textit{separating} 
if any $H\in\mathcal{G}(G)$ satisfies 
$V(H)\subseteq X$ or $V(H)\cap X$. 
\fi%%%%%%%%%%% moved close %%%%%%%%%%%%%%%%%%%%

%
%
%
%
%

%\begin{comment}
A graph is called {\em factorizable} 
if it has a perfect matching. 
A graph is called {\em  factor-critical} if 
a deletion of an arbitrary vertex results in  a factorizable graph. 
For convenience, we regard a graph with only one vertex 
as factor-critical. 
%\end{comment}

We sometimes regard a graph as the set of its vertices. 
For example, given a subgraph $H$ of $G$, 
we denote $\Gamma_{G}(V(H))$ by $\Gamma_{G}(H)$. 
For simplicity, 
regarding the operations of the contraction  
or taking the union of graphs, 
we identify vertices, edges, and subgraphs 
of the newly created graph with 
those of old graphs that naturally correspond to them.

Let $G$ be a factorizable graph. 
An edge $e\in E(G)$   is called \textit{allowed} if 
there is a perfect matching of $G$ containing $e$, and 
each connected component of the subgraph of $G$
determined by the set of all the allowed edges
is called an \textit{elementary component} of $G$. 
A factorizable graph which has exactly one elementary component 
is called \textit{elementary}. 
For an elementary component $H$ of $G$, 
we call $G[V(H)]$ a {\em factor-connected component} of $G$, 
and denote the set of all factor-connected components of $G$ by $\mathcal{G}(G)$. 
Hence, any factor-connected component is elementary and 
a factorizable graph is composed of 
its factor-connected components and additional edges 
joining distinct factor-connected components.

\subsection{The Gallai-Edmonds Partition}\label{sec:gallai-edmonds}
%cathedral_sec_gallai-edmonds
Given a graph $G$, 
we define $D(G)$ as 
\begin{quotation}
$D(G):= \{ v\in V(G) : \mbox{there is a maximum matching that exposes } v\}$. 
\end{quotation}
We also define  
$A(G)$ as $\Gamma_{G}(D(G))$  
and $C(G)$ as %$V(G)\setminus D(G)\setminus A(G)$.
$V(G)\setminus (D(G)\cup A(G))$. 
We call in this paper this partition of 
$V(G)=D(G)\dot{\cup} A(G) \dot{\cup} C(G) $ into three parts 
the {\em Gallai-Edmonds partition}. 
It is known as the {\em Gallai-Edmonds structure theorem} 
that $A(G)$ forms a  barrier with special properties, 
which is one of the most powerful theorem in matching theory~\cite{lp1986}. 
%$D(G)$, $A(G)$ and $C(G)$ are also satisfying the 
%following property regarding alternating paths. 
%

In this section, 
we present 
a proposition which shows another property  of the Gallai-Edmonds partition 
that is different from 
the Gallai-Edmonds structure theorem. 
This proposition is a well-known fact that 
connects the Gallai-Edmonds structure theorem 
and Edmonds' maximum matching algorithm,  
 and we can find it in \cite{kv2008, cc2005}. 
However, this proposition can be proved in an elementary way 
without using them, nor the notion of barriers. 
In the following 
we present it with a proof to confirm it.  
Note that Proposition~\ref{prop:gallai-edmonds} 
itself is NOT the Gallai-Edmonds structure theorem.

%%%%%%%%%%%%%%%%%%%%%%%%%%%%%%%%%%%%
\begin{proposition}\label{prop:gallai-edmonds}
Let $G$ be a graph, $M$ be a maximum matching of $G$,  
and $S$ be the set of vertices that are exposed by $M$. 
Then, the following hold:  
\renewcommand{\labelenumi}{\theenumi}
\renewcommand{\labelenumi}{{\rm \theenumi}}
\renewcommand{\theenumi}{(\roman{enumi})}
\begin{enumerate}
\item \label{item:d}
A vertex $u$ is in $D(G)$  if and only if there exists $v\in S$  such that 
there is an $M$-balanced path from $u$ to $v$. 
%$D(G) = \{ u\in V(G): \exists v\in S \mbox{ s.t. 
%there is an } M\mbox{-balanced path from } u \mbox{ to } v \}$
\item \label{item:a}
A vertex $u$ is in $A(G)$ if and only if 
there is no $M$-balanced path from $u$ to any vertex of $S$, while
there exists $v\in S$ such that 
there is an $M$-exposed path between $u$ and $v$. 
%$A(G) = \{ u\in V(G) : 
%\if0
%\mbox{ for any } v\in S 
%\mbox{ there is no } 
%M\mbox{-balanced path from } u \mbox{ to } v\fi
%u\not\in D(G)\mbox{,  
%and there is an } M\mbox{-exposed path between } u 
%\mbox{ and } v \}$
\item \label{item:c}
A vertex $u$ is in $C(G)$ if and only if for any $v\in S$
 there is neither an $M$-balanced path from 
any $u$ to $v$ nor an $M$-exposed path between $u$ and $v$. 
\if0%%%%%%%%%%%%%%%%%%%%% 
$u\in C(G)$ if and only if there is no $M$-alternating path 
whose end vertices are $u$ and any vertex in $S$.
\fi%%%%%%%%%%%%%%%%%%%%%% 
% $C(G) = \{ u\in V(G) : 
%\mbox{ for any } v\in S 
%\mbox{ there is no } M\mbox{-alternating path between }
%u \mbox{ and } v \}$ 
\end{enumerate}
\end{proposition}
\begin{proof}
For the necessity part of \ref{item:d}, 
let $P$ be the $M$-balanced path from $u$ to $v$. 
Then, $M\Delta E(P)$ is a maximum matching of $G$ 
that exposes $u$. Thus, $u\in D(G)$. 

Now we move on to the sufficiency part of \ref{item:d}. 
%For the sufficiency part of \ref{item:d}, 
If $u\in D(G)\cap S$, 
the trivial $M$-balanced path $(\{u\}, \emptyset )$ 
satisfies the property. 
Otherwise, that is, if $u \in D(G)\setminus S$, 
by the definition of $D(G)$ 
there is a maximum matching $M'$ of $G$ that exposes $u$. 
Then, $G. M\Delta M'$ has a connected component 
which is an $M$-balanced path 
from $u$ to some vertex in $S$. 
Hence, we are done for \ref{item:d}. 
%%
%\qed\end{proof}
%
%By Claim~\ref{claim:d} we obtain \ref{item:d}. 

%\begin{proof}
For \ref{item:a}, we first prove the necessity part. 
Let $P$ be the $M$-exposed path between $u$ and $v$, 
and $w\in V(P)$ be such that $uw\in E(P)$. 
Then, $P - u$ is an $M$-balanced path from $w$ to $v$, 
which means $w\in D(G)$ by \ref{item:d}. 
Then, we have $u\in A(G)$, 
since the first part of the condition on $P$ yields 
$u\not\in D(G)$ by \ref{item:d}. 

Now we move on to the sufficiency part of \ref{item:a}. 
Note that the first part of the conclusion follows by \ref{item:d}. 
By the definition of $A(G)$, 
there exists $w \in D(G)$ such that $wu\in E(G)$. 
By \ref{item:d}, there is an $M$-balanced path $Q$ from $w$ to a vertex $v\in S$. 
If $u \in V(Q)$, 
then since $u\not\in D(G)$, 
the subpath of $Q$ from $v$ to $u$ 
is an $M$-exposed path between $v$ and $u$ by \ref{item:d}.   
Thus,  the claim follows. 
Otherwise, that is, if $u\not\in V(Q)$, 
then $Q + wu$ forms an $M$-exposed path between $v$ and $u$.   
Therefore,  again the claim follows. Thus, we are done for \ref{item:a}.  
%%
%\qed\end{proof}

Since we obtain \ref{item:d} and \ref{item:a}, 
consequently \ref{item:c} follows. 
\qed\end{proof}  
%
%
%\begin{comment}
%The next proposition is easy to see. 
%%
%\begin{proposition}\label{prop:deleteone}
%Let $G$ be a factorizable graph, $M$ be a perfect matching of $G$, 
%and $x\in V(G)$. Let $x'\in V(G)$ be such that $xx'\in M$. 
%Then, $M\setminus \{xx'\}$ is a maximum matching of $G-x$, 
%exposing only $x'$. 
%\end{proposition}
%\end{comment}
%
%
The next proposition is also known (see \cite{cc2005})  
and is easily obtained from Proposition~\ref{prop:gallai-edmonds}. 
%
\if0%%%%%%%%
By Proposition~\ref{prop:gallai-edmonds}, 
the following can be easily obtained. 
The property itself might be a folklore, 
however note that it can also be derived by quite an elementary way 
and independent from the Gallai-Edmonds structure theorem. 
\fi%%%%%%%%%%
\begin{proposition} \label{prop:factorizable-ge}
Let $G$ be a factorizable graph and $M$ be a perfect matching of $G$. 
Then, for any $x\in V(G)$, 
 the following hold:  
\renewcommand{\labelenumi}{\theenumi}
\renewcommand{\labelenumi}{{\rm \theenumi}}
\renewcommand{\theenumi}{(\roman{enumi})}
\begin{enumerate}
\item \label{item:f-d}
A vertex $u$ is in $D(G-x)$ if and only if 
there is an $M$-saturated path between $x$ and $u$. 
\item \label{item:f-a}
A vertex $u$ is in $A(G-x)\cup\{x\}$ if and only if 
there is no $M$-saturated path between $x$ and $u$,  
while there is an $M$-balanced path from $x$ to $u$. 
\item \label{item:f-c}
A vertex $u$ is in $C(G-x)$ if and only if 
there is neither an  $M$-saturated path between $u$ and $x$ 
nor an $M$-balanced path from $x$ to $u$. 
\end{enumerate}
\end{proposition}
\begin{proof}
Let $x'\in V(G)$ be such that $xx'\in M$. 
Let $G':= G -x$ and $M':= M \setminus \{xx'\}$. 
Note that apparently 
\begin{quotation} 
$M'$ is a maximum matching of $G'$, 
exposing only $x'$.
\end{quotation} 
By Propositions~\ref{prop:gallai-edmonds}, 
$u\in D(G')$ if and only if 
there is an $M'$-balanced path from $u$ to $x'$. 
Additionally, the following apparently holds:  
there is an $M'$-balanced path from $u$ to $x'$ in $G'$
if and only if 
there is an $M$-saturated path between $u$ and $x$ in $G$. 
Thus, we obtain \ref{item:f-d}. 
The other claims, \ref{item:f-a} and \ref{item:f-c}, 
also follow by similar arguments. 
\qed\end{proof}
%
%
\if0%%%%%%%%% moved %%%%%%%%%%%%%%%%%%%%%%%%%%%
\begin{theorem}\label{thm:ge2cathedral}
Let $G$ be a factorizalbe graph such that 
the poset $(\mathcal{G}(G), \yield)$ has the minimum element $G_0$. 
Then, $V(G_0)$ is the set of vertices 
that are never contained in $C(G-x)$ for any $x\in V(G)$. 
\end{theorem}
\fi%%%%%%%%%%%%%%%%%%%%%%%%%%%%%%%%%%%%%%%%%%%%

Proposition~\ref{prop:factorizable-ge} associates 
factorizable graphs with the Gallai-Edmonds partition, and   
 it will be used later in the proof of  Theorem~\ref{thm:ge2cathedral}.   
Hence it will contribute to the new proof of the cathedral theorem. 

\begin{comment}
Proposition~\ref{prop:factorizable-ge} associates 
factorizable graphs with the Gallai-Edmonds partition;  
hence it will be used later in proving Theorem~\ref{thm:ge2cathedral} 
and so will contribute to the new proof of the cathedral theorem. 
\end{comment}

\section{Outline of the New Proof} \label{sec:outline}

Here we give an outline of how we give a new proof of the cathedral theorem 
together with backgrounds of the theorem.  
In our previous work~\cite{kita2012a, kita2012b}, we revealed canonical structures 
of factorizable graphs. 
The key points of them are as follows.  
(We shall explain them in detail in Section~\ref{sec:canonical}.) 
\begin{enumerate}
\renewcommand{\labelenumi}{\theenumi}
\renewcommand{\labelenumi}{{\rm \theenumi}}
\renewcommand{\theenumi}{(\alph{enumi})}
\item \label{item:order}
For a factorizable graph $G$, 
a partial order $\yield$ can be defined on the factor-connected components $\mathcal{G}(G)$ (Theorem~\ref{thm:order}). 
\item \label{item:partition}%\label{item:related}
An equivalence relation $\gsim{G}$ based on factor-connected components can be defined on $V(G)$ (Theorem~\ref{thm:generalizedcanonicalpartition}). 
The equivalence classes by $\gsim{G}$ can be regarded as a generalization of Kotzig's canonical partition~\cite{kotzig1959a, kotzig1959b, kotzig1960}. 
\item \label{item:related} 
These two notions $\yield$ and $\gsim{G}$ are related each other 
in the sense that 
for  $H\in\mathcal{G}(G)$ 
a relationship between $H$ and its strict upper bounds in the poset $(\mathcal{G}(G), \yield)$   
can be described using $\gsim{G}$  
(Theorem~\ref{thm:base}).  
\end{enumerate}  
In Section~\ref{sec:alternating}, 
we begin to present new results in this paper.  
We further consider the structures given by \ref{item:order} \ref{item:partition} \ref{item:related} and show a relationship between the structures and the Gallai-Edmonds partition:   
%We investigate a relationship between the structures \ref{item:order} \ref{item:partition} \ref{item:related} 
%and the Gallai-Edmonds partition  and obtain that 
\begin{quote}  
If the poset $(\mathcal{G}(G), \yield)$ of a factorizable graph $G$ 
has the minimum element $G_0$,   
then $V(G_0) = V(G)\setminus \bigcup_{x\in V(G)} C(G-x)$ (Theorem~\ref{thm:ge2cathedral}).  
\end{quote} 
This theorem later plays a crucial role
in the new proof of the cathedral theorem.

In Section~\ref{sec:proof}, we consider saturated graphs and present a new proof of the cathedral theorem. 
Given a saturated elementary graph and a family of saturated graphs satisfying a certain condition, we can define an operation, the {\em cathedral construction},  
that creates a new graph obtained from the given graphs by adding new edges.  
Here  the given graphs are 
called the {\em foundation} and the family of {\em towers}, respectively.  
We consider the structures by \ref{item:order} \ref{item:partition} \ref{item:related} for saturated graphs and obtain the following: 
\begin{quote}
If $G$ is a saturated graph, 
then the poset $(\mathcal{G}(G), \yield)$ has the minimum element $G_0$ (Lemma~\ref{lem:minimum}).  \end{quote}
Moreover, $G_0$ and all connected components of $G-V(G_0)$ are saturated 
and they are well-defined as a foundation and towers (Lemmas~\ref{lem:gpartispart} and \ref{lem:each_saturated}). 
We show that 
$G$ is the graph obtained from them by the cathedral construction (Theorem~\ref{thm:cathedral_nec}). 
%$G$ is the graph obtained by conducting the cathedral construction to them (Theorem~\ref{thm:cathedral_nec}). 
%
\begin{quote}
Conversely, 
if a graph $G$ obtained by the cathedral construction 
from a foundation $G_0$ and some towers is saturated, 
and $G_0$ is the minimum element of the poset $(\mathcal{G}(G), \yield)$ (Theorem~\ref{thm:cathedral_suff}). 
\end{quote}
By Theorems~\ref{thm:cathedral_nec} and \ref{thm:cathedral_suff}, the constructive characterization of the saturated graphs---the most important part of the cathedral theorem---is obtained. Additionally, the other parts of the cathedral theorem follow quite smoothly by Theorem~\ref{thm:ge2cathedral} and the natures of the structures given by \ref{item:order} \ref{item:partition} \ref{item:related}.

\section{Canonical Structures of Factorizable Graphs}\label{sec:canonical}

In this section we introduce canonical structures of factorizable graphs, 
which will later turn out to be the underlying structure 
of  the cathedral theorem. 
They are composed mainly of three parts: 
a partial order on factor-connected components (Theorem~\ref{thm:order}), 
a generalization of the canonical partition (Theorem~\ref{thm:generalizedcanonicalpartition}), 
and a relationship between them (Theorem~\ref{thm:base}). 
\begin{definition}\label{def:separating}
Let $G$ be a factorizable graph. 
A set $X\subseteq V(G)$ is {\em separating} 
if any $H\in\mathcal{G}(G)$ satisfies 
$V(H)\subseteq X$ or $V(H)\cap X = \emptyset$.  
\end{definition} 

It is easy to see that  
the following four statements are equivalent for 
a factorizable graph $G$ and $X \subseteq V(G)$:  
%Given a factorizable graph $G$ and $X\subseteq V(G)$, the following are equivalent: 
\begin{enumerate}
\renewcommand{\labelenumi}{\theenumi}
\renewcommand{\labelenumi}{{\rm \theenumi}}
\renewcommand{\theenumi}{(\roman{enumi})}
\item 
The set $X$ is separating. 
\item 
Either $X = \emptyset$  or there exist $H_1,\ldots, H_k\in\mathcal{G}(G)$ 
with $X = V(H_1)\dot{\cup} \cdots \dot{\cup} V(H_k)$. 
\item 
For any perfect matching $M$ of $G$, 
$M$ contains a perfect matching of $G[X]$. 
\item 
For any perfect matching $M$ of $G$, 
$\delta_{G}(X)\cap M = \emptyset$. 
\end{enumerate}

\begin{definition}
Let $G$ be a factorizable graph. 
We define a binary relation $\yield$ on $\mathcal{G}(G)$ as follows: 
For $G_1, G_2\in \mathcal{G}(G)$, 
$G_1\yield G_2$ if there exists $X\subseteq V(G)$ such that 
\begin{comment}
Let $G$ be a factorizable graph, 
and let $G_1, G_2\in\mathcal{G}(G)$. 
We say $G_1\yield G_2$ if there exists $X\subseteq V(G)$ such that
\end{comment} 
\begin{enumerate}
\item $X$ is separating, 
\item $V(G_1)\cup V(G_2)\subseteq X$, and 
\item $G[X]/V(G_1)$ is factor-critical. 
\end{enumerate}
\end{definition}

For the above relation, the following theorem is known: 

\begin{theorem}[Kita~\cite{kita2012a, kita2012b}]\label{thm:order}
For any factorizable graph $G$, 
the binary relation $\yield$ is a partial order on $\mathcal{G}(G)$. 
\end{theorem}

\begin{definition}
Let $G$ be a factorizable graph. 
We define a binary relation $\gsim{G}$ on $V(G)$ as follows: 
For $u,v\in V(G)$, 
$u\gsim{G} v$  if   
\begin{comment}
Let $G$ be a factorizable graph, 
and let $u,v\in V(G)$. 
We say $u\gsim{G} v$ 
\end{comment}  
%
\begin{enumerate}
\item 
$u$ and $v$ are contained in the same factor-connected component 
and 
\item 
either 
%satisfy either of the following: 
$u$ and $v$ are identical,  or $G-u-v$ is not factorizable.  
\end{enumerate}
\end{definition}

We also have a theorem for the relation $\gsim{G}$: 

\begin{theorem}[Kita~\cite{kita2012a, kita2012b}]\label{thm:generalizedcanonicalpartition}
For any factorizable graph $G$, 
the binary relation $\gsim{G}$ is an equivalence relation on $V(G)$. 
\end{theorem}
If a graph $G$ is elementary, then 
the family of equivalence classes by $\gsim{G}$, i.e., $V(G)/\gsim{G}$  coincides with 
Kotzig's {\em canonical partition}%
~\cite{lp1986, kotzig1959a, kotzig1959b, kotzig1960} (see \cite{kita2012a, kita2012b}).  
Therefore, 
given a factorizable graph $G$,  
we call $V(G)/\gsim{G}$  
the {\em generalized canonical partition}, 
and denote it by $\gpart{G}$. 
By the definition of $\gsim{G}$, 
each member of $\gpart{G}$ 
is  contained in some factor-connected component.
%
%
\begin{comment} 
Given a factorizable graph $G$, 
we call the family of equivalence classes by $\gsim{G}$ 
the {\em generalized canonical partition}, 
and denote $\gpart{G} := V(G)/\gsim{G}$, 
since it coincides with Kotzig's {\em canonical partition}%
~\cite{lp1986, kotzig1959a, kotzig1959b, kotzig1960} 
if $G$ is elementary~\cite{kita2012a, kita2012b}. 
By the definition of $\gsim{G}$, 
each member of $\gpart{G}$ 
is respectively contained in a factor-connected component.
\end{comment} 
%
\begin{comment}
As we see from the definition of $\gsim{G}$, 
each member of $\gpart{G}$ 
is respectively contained in a factor-connected component. 
\end{comment}
Therefore,  
$\pargpart{G}{H} := \{ S\in \gpart{G}: S\subseteq V(H)\}$ 
forms a partition of $V(H)$ for each $H\in\mathcal{G}(G)$. 

Note also the following, 
 which are also stated in \cite{kita2012a, kita2012b}:   
\begin{fact}\label{fact:refinement}
Let $G$ be a factorizable graph, and let $H\in\mathcal{G}(G)$. 
Then, $\pargpart{G}{H}$ is a refinement of $\gpart{H} = \pargpart{H}{H}$; 
that is,  
if $u,v\in V(H)$ satisfies $u\gsim{G} v$, then $u\gsim{H} v$ holds. 
\end{fact} 
%cathedral_proof_prop_refinement
\begin{proof}
We prove the contrapositive. 
Let $u, v\in V(H)$ be such that 
$u \not \gsim{H} v$, 
which is equivalent to $u$ and $v$ satisfying 
$u\neq v$ and $H-u-v$ is factorizable. 
Let $M$ be a perfect matching of $H-u-v$. 
Since $G-V(H)$ is also factorizable,  
 by letting $M'$ be a perfect matching of it, 
we can construct a perfect matching of $G-u-v$, 
namely,  $M \cup M'$. 
Therefore, $u \not \gsim{G} v$. 
%\qed
\qed\end{proof}

The following fact can be immediately obtained by Property~\ref{prop:deletable}.

\begin{fact}\label{fact:deletable2path}
Let $G$ be a factorizable graph, and $M$ be a perfect matching of $G$. 
Let $u, v\in V(G)$ be vertices contained in the 
same factor-connected component of $G$. 
Then, 
$u\gsim{G} v$ if and only if 
there is no $M$-saturated path between $u$ and $v$. 
\end{fact}
%\begin{proof}
%This is immediate by Property~\ref{prop:deletable}. 
%%
%\qed\end{proof}
%

\begin{definition}
Let $G$ be a factorizable graph, and let $H\in\mathcal{G}(G)$. 
We denote the upper bounds of $H$ in the poset $(\mathcal{G}(G), \yield)$ 
by $\parupstar{G}{H}$; 
that is,  
%\begin{quotation}
$\parupstar{G}{H} := \{ H'\in\mathcal{G}(G): H\yield H'\}$. 
%\end{quotation} 
We define 
$\parup{G}{H} := \parupstar{G}{H}\setminus \{H\}$, 
and the vertices contained in $\parupstar{G}{H}$ (resp. $\parup{G}{H}$ ) 
as $\vparupstar{G}{H}$ (resp. $\vparup{G}{H}$ ); 
i.e., 
$\vparupstar{G}{H} := \bigcup_{H'\in\parupstar{G}{H}} V(H')$ and 
$\vparup{G}{H} := \bigcup_{H'\in\parup{G}{H}} V(H')$.  
We often omit the subscripts ``$G$'' 
if they are apparent from contexts. 
\end{definition}
There is a relationship between the partial order and the generalized canonical partition:  
%The next theorem describes a relationship between the partial order and the generalized canonical partition. 
%The next theorem states an interrelationship between the partial order 
%and the generalized canonical partition. 
%
\begin{theorem}[Kita~\cite{kita2012a, kita2012b}]\label{thm:base}
Let $G$ be a factorizable graph, 
and let $H\in\mathcal{G}(G)$. 
Let $K$ be one of the connected components of $G[\vup{H}]$. 
Then, 
there exists $S_K\in\pargpart{G}{H}$ such that 
$\Gamma_{G}(K)\cap V(H)\subseteq S_K$. 
\end{theorem} 
In the above theorem, 
$\vup{H}$ and $\pargpart{G}{H}$ are notions determined by $\yield$ and $\gsim{G}$, respectively.  
Therefore, Theorem~\ref{thm:base} describes a relationship between 
$\yield$ and $\gsim{G}$.  
Let us add some propositions used later in this paper. 
\begin{proposition}[Kita~\cite{kita2012a, kita2012b}]\label{prop:ear-base}
Let $G$ be a factorizable graph and $M$ be a perfect matching of $G$,  
and let $H\in\mathcal{G}(G)$. 
Then, for any $M$-ear $P$ relative  to $H$, 
its end vertices $u,v\in V(H)$ satisfy 
$u\gsim{G} v$. 
\end{proposition}

\begin{theorem}[Kita~\cite{kita2012a, kita2012b}]\label{thm:add}
Let $G$ be a factorizable graph, 
$G_1\in\mathcal{G}(G)$ be a minimal element of the poset $(\mathcal{G}(G), \yield)$, 
and $G_2\in\mathcal{G}(G)$ be such that $G_1\yield G_2$ does not hold. 
Then, 
$G$ has (possibly identical) complement edges $e, f$ 
joining $V(G_1)$ and $V(G_2)$ 
such that 
$\mathcal{G}(G + e+ f) = \mathcal{G}(G)$ 
and $G_1\yield G_2$ in $(\mathcal{G}(G+e+f), \yield)$. 
\end{theorem}
Theorem~\ref{thm:add} will play a crucial role in Section~\ref{sec:proof} 
when we show that the poset by $\yield$ has the minimum element if a given graph is saturated (Lemma~\ref{lem:minimum}).

\section{Factorizable Graphs through the Gallai-Edmonds Partition}\label{sec:alternating}

%cathedral_sec_alternating
In this section, 
we present a new result on a relationship between 
the Gallai-Edmonds partition and 
the canonical structures of factorizable graphs in Section~\ref{sec:canonical}. 
As we later see in Section~\ref{sec:proof}, 
Theorem~\ref{thm:ge2cathedral} can be regarded as a generalization of 
a part of the statements of the cathedral theorem. 

\begin{theorem}\label{thm:ge2cathedral}
Let $G$ be a factorizable graph such that 
the poset $(\mathcal{G}(G), \yield)$ has the minimum element $G_0$. 
Then, $V(G_0)$ is exactly the set of vertices 
that is disjoint from $C(G-x)$ for any $x\in V(G)$; 
that is, $V(G_0) = V(G)\setminus \bigcup_{x\in V(G)} C(G-x)$. 
\end{theorem}

To show Theorem~\ref{thm:ge2cathedral}, we give some definitions and lemmas. 
Let $G$ be a factorizable graph, 
and let $H\in\mathcal{G}(G)$ and $S\in\pargpart{G}{H}$. 
Based on Theorem~\ref{thm:base}, 
we denote the set of all the strict upper bounds of $H$ 
``assigned'' to $S$ by $\parup{G}{S}$;  
that is to say, 
$H'\in \parup{G}{S}$ if and only if 
$H'\in\up{H}$ and 
there is a connected component $K$ of $G[\vup{H}]$ 
such that $V(H')\subseteq V(K)$ and $\Gamma_{G}(K)\cap V(H)\subseteq S$. 
%%%%%%%%%%%%%%%%%
\begin{comment}
Namely, 
$\parup{G}{S}: = \{ H'\in \up{H}: 
\mbox{ there is a connected component } K \mbox{ of } G[\vup{H}]  
\mbox{ such that } V(H')\subseteq V(K) \mbox{ and }  \Gamma_{G}(K)\cap V(H)\subseteq S \}$. 
\end{comment}
%%%%%%%%%%%%%%%%
We define 
$\vparup{G}{S} := \bigcup_{H'\in\parup{G}{S}} V(H')$ 
and $\vparupstar{G}{S} := \vparup{G}{S}\cup S$. 
We often omit the subscripts ``$G$'' if they are apparent 
from contexts.    
Note that $\up{H} = \bigcup_{S\in\pargpart{G}{H}} \up{S}$.  

Next, we present fundamental results on factorizable graphs.

\begin{proposition}[Kita~\cite{kita2012a, kita2012b}]\label{prop:nonpositive}
If $H$ is an elementary graph, 
then for any $u, v\in V(H)$ 
there is an $M$-saturated path between $u$ and $v$, 
or there is an $M$-balanced path from $u$ to $v$,  
where $M$ is an arbitrary perfect matching of $H$. 
\end{proposition}

\begin{lemma}[Kita~\cite{kita2012f, kita2013a}]\label{lem:reach}
Let $G$ be a factorizable graph, and 
$M$ be a perfect matching of $G$. 
Let $H\in\mathcal{G}(G)$, $S\in\pargpart{G}{H}$, 
and $T\in\pargpart{G}{H}\setminus \{S\}$. 
\renewcommand{\labelenumi}{\theenumi}
\renewcommand{\labelenumi}{{\rm \theenumi}}
\renewcommand{\theenumi}{(\roman{enumi})}
\begin{enumerate}
\item \label{item:up-path} 
For any $u\in \vupstar{S}$, 
there is an $M$-balanced path from $u$ to some vertex $v\in S$ 
whose vertices except $v$ are in $\vup{S}$. 

\item \label{item:path2up}
For any $u\in S$ and $v\in \vupstar{T}$, 
there is an $M$-saturated path between $u$ and $v$ 
whose vertices are all contained in $\vupstar{H}\setminus \vup{S}$. 

\item \label{item:nosaturate}
For any $u\in S$ and $v\in \vup{S}$, 
there are neither $M$-saturated paths between $u$ and $v$ 
nor $M$-balanced paths from $u$ to $v$. 

\item \label{item:nonpositive}
For any $u, v\in S$, 
there is no $M$-saturated path between $u$ and $v$, 
while there is an $M$-balanced path from $u$ to $v$. 
\end{enumerate}
\end{lemma} 
\begin{proof}
The statements \ref{item:up-path}, \ref{item:path2up}, and \ref{item:nosaturate} 
are stated in \cite{kita2012f, kita2013a}. 
The statement \ref{item:nonpositive} 
is immediately obtained by combining Fact~\ref{fact:deletable2path}  
and Proposition~\ref{prop:nonpositive}. 
\qed\end{proof}
By Proposition~\ref{prop:nonpositive} and Lemma~\ref{lem:reach},  
the next lemma follows. 
\begin{lemma}\label{prop:combination}\label{lem:combination}
%\begin{proposition}\label{prop:combination}
Let $G$ be a factorizable graph, 
and $M$ be a perfect matching of $G$. 
Let $H\in \mathcal{G}(G)$  and $S\in \pargpart{G}{H}$. 
Then, the following hold:  
\begin{enumerate}
\renewcommand{\labelenumi}{\theenumi}
\renewcommand{\labelenumi}{{\rm \theenumi}}
\renewcommand{\theenumi}{(\roman{enumi})}
\item \label{item:up2up}
%Let $u\in \vup{S}$. 
For any $u\in \vup{S}$ and $v\in \vupstar{H}\setminus \vupstar{S}$, 
there is an $M$-saturated path between $u$ and $v$. 
\item \label{item:up2root}
For any $u\in \vup{S}$ and $v\in S$, 
there is no $M$-saturated path between $u$ and $v$;  
however,  there is an $M$-balanced path from $u$ to $v$. 
\item \label{item:root2up}
%Let $u\in S$. 
For any $w\in S$ and $v\in \vupstar{H}\setminus \vupstar{S}$, 
there is an $M$-saturated path between $w$ and $v$. 
\item \label{item:root2root}
For any $w, v\in S$, 
there is no $M$-saturated path between $w$ and $v$;  
however,  there is an $M$-balanced path from $w$ and $v$. 
\item \label{item:inverse}
For any $w\in S$ and $v\in \vup{S}$, 
there is neither an $M$-saturated path between $w$ and $v$ 
nor an $M$-balanced path from $w$ to $v$. 
\end{enumerate}
\end{lemma}
\begin{proof}
The statements \ref{item:root2up}, \ref{item:root2root}, and \ref{item:inverse} 
are immediate from  \ref{item:path2up}, 
\ref{item:nonpositive}, and \ref{item:nosaturate} of Lemma~\ref{lem:reach}, respectively. 

For \ref{item:up2up}, 
let $P_1$ be an $M$-balanced path from $u$ to 
some vertex $x\in S$ such that $V(P_1)\setminus \{x\} \subseteq \vup{S}$, 
given by \ref{item:up-path} of Lemma~\ref{lem:reach}. 
By \ref{item:path2up} of Lemma~\ref{lem:reach}, 
there is an $M$-saturated path $P_2$ between $x$ and $v$ 
such that $V(P_2)\subseteq \vupstar{H}\setminus \vup{S}$. 
Hence, 
the path obtained by adding $P_1$ and $P_2$ 
 forms an $M$-saturated path between $u$ and $v$, 
and \ref{item:up2up} follows. 

The first and the latter halves of \ref{item:up2root} 
are restatements of 
\ref{item:nosaturate} and \ref{item:up-path} of Lemma~\ref{lem:reach}, respectively. 
%%%%%%%%%%%%%%%%%%%%%%%%%%%%%%%%%%%%%%%%%%%%%%%%%%%%%%
\begin{comment}
First note the following claim; 
\begin{cclaim}\label{claim:balanced}
For any $x, y\in S$, 
there is no $M$-saturated path in $G$ between $x$ and $y$, and   
there is an $M$-balanced path from $x$ to $y$ 
which is contained in $H$. 
\end{cclaim}
\begin{proof}
The first part of the claim is immediate
 by Proposition~\ref{prop:deletable2path}. 
Thus, the latter part follows by Proposition~\ref{prop:nonpositive}. 
%Obvious from Proposition~\ref{prop:nonpositive}. 
%
\qed\end{proof}
%
\ref{item:up2up} 
is obtained easily by Lemmas~\ref{lem:up-path} and \ref{lem:path2up}. 
\ref{item:up2root}
is obtained easily by Lemma~\ref{lem:up-path} and Claim~\ref{claim:balanced}. 
\ref{item:root2up} 
is immediate from Lemma~\ref{lem:path2up}, 
and \ref{item:root2root} 
is rather immediate from Claim~\ref{claim:balanced}.   
\ref{item:inverse} 
is immediate from Lemma~\ref{lem:nosaturate}. 
\if0%%%%%%%%%%%%%%%%%%%%%%%%%%%%%%
For \ref{item:up2up},
by Lemma~\ref{lem:up-path},  
there is an $M$-balanced path $P_1$ from $u$ to a vertex $x\in S$
which is contained in $\vup{S}\cup\{x\}$. 
By Claim~\ref{claim:balanced}, 
there is an $M$-balanced path $P_2$ from $x$ to $v$ 
which is contained in $V(H)$. 
Thus, $P_1 + P_2$ is a desired $M$-balanced path, 
and we are done for \ref{item:up2up}.
\fi%%%%%%%%%%%%%%%%%%%%%%%%%%%%%%%%%%%% 
\end{comment}
%%%%%%%%%%%%%%%%%%%%%%%%%%%%%%%%%%%%%%%%%%%%%%%%%%%%%%%%%%%%%
%
\qed\end{proof}
By comparing Proposition~\ref{prop:factorizable-ge} and Lemma~\ref{lem:combination}, 
the next lemma follows.  
\begin{lemma}\label{prop:ge2path}\label{lem:ge2path}
%\begin{proposition}\label{prop:ge2path}
Let $G$ be a factorizable graph 
such that the poset $(\mathcal{G}(G), \yield )$ has the 
minimum element $G_0$. 
Let $S\in \pargpart{G}{G_0}$. 
\renewcommand{\labelenumi}{\theenumi}
\renewcommand{\labelenumi}{{\rm \theenumi}}
\renewcommand{\theenumi}{(\roman{enumi})}
\begin{enumerate}
\item \label{item:supset}
If $x\in \vup{S}$, then   
$D(G-x) \supseteq \vupstar{G_0}\setminus \vupstar{S}$, 
$A(G-x)\cup\{x\} \supseteq  S$, 
and $C(G-x) \subseteq \vup{S}$. 
\item \label{item:identical} 
If $x\in S$, then 
$ D(G-x) = \vupstar{G_0}\setminus \vupstar{S}$, 
$A(G-x)\cup\{x\} = S$, 
and $C(G-x) = \vup{S}$. 
\end{enumerate}
%\end{proposition}
\end{lemma}
\begin{proof}
The claims are all obtained by comparing 
the reachabilities of alternating paths 
regarding Proposition~\ref{prop:factorizable-ge} 
and Lemma~\ref{prop:combination}.  
Let $x\in \vup{S}$. 
By Proposition~\ref{prop:factorizable-ge} \ref{item:f-d} 
and Lemma~\ref{prop:combination} \ref{item:up2up}, 
we have $ D(G-x) \supseteq \vupstar{G_0}\setminus \vupstar{S}$. 
It also follows that $A(G-x)\cup\{x\} \supseteq  S$
by a similar argument, comparing Proposition~\ref{prop:factorizable-ge} \ref{item:f-a} 
and Lemma~\ref{prop:combination} \ref{item:up2root}. 
Therefore, since $V(G) = D(G)\dot{\cup} A(G) \dot{\cup} C(G) 
= ( \vupstar{G_0}\setminus \vupstar{S})  \dot{\cup} S \dot{\cup} \vup{S}$, 
we have $C(G-x) \subseteq \vup{S}$, and we are done for \ref{item:supset}. 
The statement \ref{item:identical} also follows by similar arguments 
with Proposition~\ref{prop:factorizable-ge} and 
Lemma~\ref{prop:combination} \ref{item:root2up} \ref{item:root2root} \ref{item:inverse}.   
%
%%%%%%%%%%%%%%%%%%%%%%%%%%%%%%%%%%%%%%
\begin{comment}
This is easily obtained 
combining Propositions~\ref{prop:factorizable-ge} 
and \ref{prop:combination}. 
\end{comment}
%%%%%%%%%%%%%%%%%%%%%%%%%%%%%%%%%%%%%%%
%
\qed\end{proof} 
Now we can prove Theorem~\ref{thm:ge2cathedral} using Lemma~\ref{lem:ge2path}. 
\begin{proof}[Theorem~\ref{thm:ge2cathedral}]
\begin{cclaim}\label{claim:disjoint}
For any $x\in V(G)$, $V(G_0)\cap C(G-x) = \emptyset$. 
\end{cclaim}
\begin{proof}
Let $u\in V(G_0)$ 
and let $S\in\pargpart{G}{G_0}$ be such that $u\in S$. 
By Lemma~\ref{prop:ge2path}, if $x\in \vupstar{S}$ 
then $u\in A(G-x)$, 
and if $x\in \vupstar{G_0}\setminus \vupstar{S}$  
then $u\in D(G-x)$. 
Thus, anyway we have $u\not\in C(G-x)$,  and  the claim follows. 
\qed\end{proof}
\begin{cclaim}\label{claim:notdisjoint}
For any $u\in V(G)\setminus V(G_0)$, 
there exists $x\in V(G)$ such that $u\in C(G-x)$. 
\end{cclaim}
\begin{proof}
Let $u\in V(G)\setminus V(G_0)$ and 
let $S\in \pargpart{G}{G_0}$ be such that $u\in \vup{S}$. 
%Let $u\in \vup{S}$. 
Then, for any $x\in S$, 
we have $u\in C(G-x)$ by Lemma~\ref{prop:ge2path}. 
Thus, we have the claim. 
\qed\end{proof}
By Claims~\ref{claim:disjoint} and \ref{claim:notdisjoint}, 
we obtain the theorem. 
%
%%%%%%%%%%%%%%%%%%%%%%%%%%%%%%%%%%%
\begin{comment}
Let $u\in V(G_0)$ 
and let $S\in\pargpart{G}{G_0}$ be such that $u\in S$. 
By Proposition~\ref{prop:ge2path}, if $x\in \vupstar{S}$, 
then $u\in A(G-x)$, 
and if $x\in \vupstar{G_0}\setminus \vupstar{S}$, 
then $u\in D(G-x)$. 
Thus, we have $V(G_0)\cap C(G-x) = \emptyset $ 
for any $x\in V(G)$. 

Let $u\in V(G)\setminus V(G_0)$ and 
let $S\in \pargpart{G}{G_0}$ be such that $u\in \vup{S}$. 
%Let $u\in \vup{S}$. 
Then, for any $x\in S$, 
$u\in C(G-x)$ by Proposition~\ref{prop:ge2path}. 
Thus, we have the claim. 
\end{comment}
%%%%%%%%%%%%%%%%%%%%%%%%%%%%%%%%%%%%
%
\qed\end{proof}
As we mentioned in the outline given in Section~\ref{sec:outline}, 
we will obtain in Section~\ref{sec:proof} that if a graph is saturated then the poset by $\yield$ has the minimum element.  
Thus,  the above theorem, Theorem~\ref{thm:ge2cathedral}, 
will turn out to be regarded as a generalized version of the part of the cathedral theorem 
related to the Gallai-Edmonds partition.

\section{Another Proof of the Cathedral Theorem}\label{sec:proof}

\subsection{The Cathedral Theorem}

\begin{comment}
A graph with perfect matchings is called {\em  saturated} if 
an addition of an arbitrary complement edge creates a new perfect matching. 
\end{comment}
The {\em cathedral theorem} is a structure theorem of saturated graphs, 
originally given by Lov\'asz~\cite{lovasz1972b, lp1986}, 
and later Szigeti gave another proof~\cite{szigeti1993, szigeti2001}. 
In this section,  we give yet another proof  
 as a consequence of the structures given in Section~\ref{sec:canonical}. 
 For convenience, we treat empty graphs 
 as factorizable and saturated.

\begin{definition}[The Cathedral Construction]\label{def:cathedralconstruction} 
Let $G_0$ be a saturated elementary graph 
and let $\{G_S\}_{S\in\gpart{G_0}}$ be a family of saturated graphs, 
some of which might be empty. 
For each $S\in\mathcal{P}(G_0)$, 
join every vertex in $S$ and every vertex of $G_S$.
We call this operation the \textit{cathedral construction}. 
Here $G_0$ and $\{G_S\}_{S\in\gpart{G_0}}$ 
are respectively called  
the \textit{foundation} and the family of \textit{towers}. 
% 
\if0%%%%%%%%%%%%%%%%%%%%%%%%%%%%%%%%%%%%%%%%%%%%%
Let $G_0$ be a saturated elementary graph.
For each class $S\in\mathcal{P}(G_0)$, assign an saturated graph $G_S$, which might be empty,
and join every vertex in $S$ and every vertex of $G_S$.
We call this operation \textit{cathedral construction}.
\fi%%%%%%%%%%%%%%%%%%%%%%%%%%%%%%%%%%%%%%%%%%%%%%%%%
\end{definition}
%

%
%\begin{center}
\begin{figure}
\centering
%\begin{center}
\includegraphics[width=7cm]{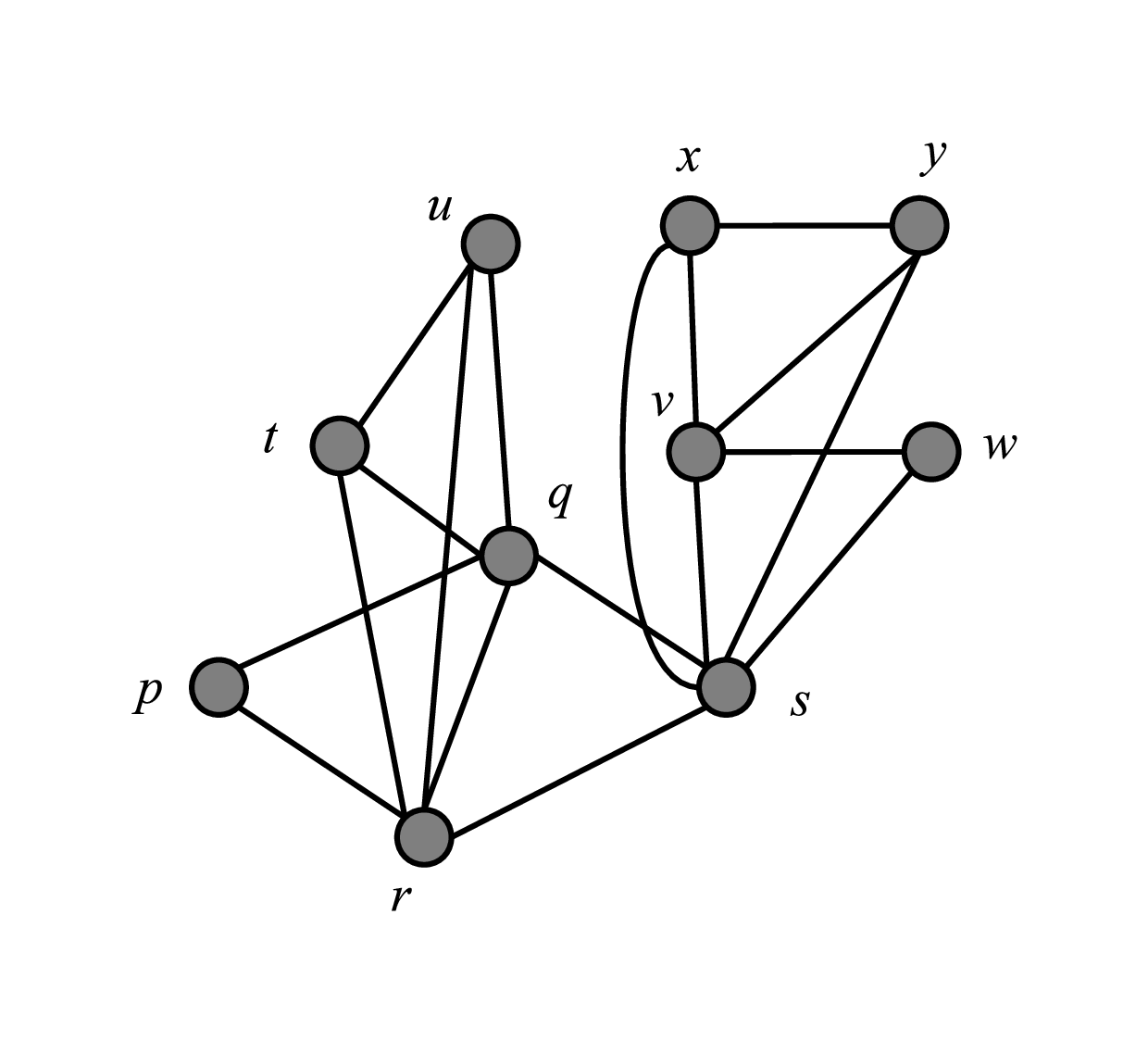}
%\end{center}
\caption{A saturated graph $\tilde{G}$}
\label{fig:saturated_verticesnumbered}
%\end{center}
\end{figure}
%\end{center}

\begin{figure}
\begin{minipage}{0.58\hsize}
\begin{center}
\includegraphics[width=7cm]{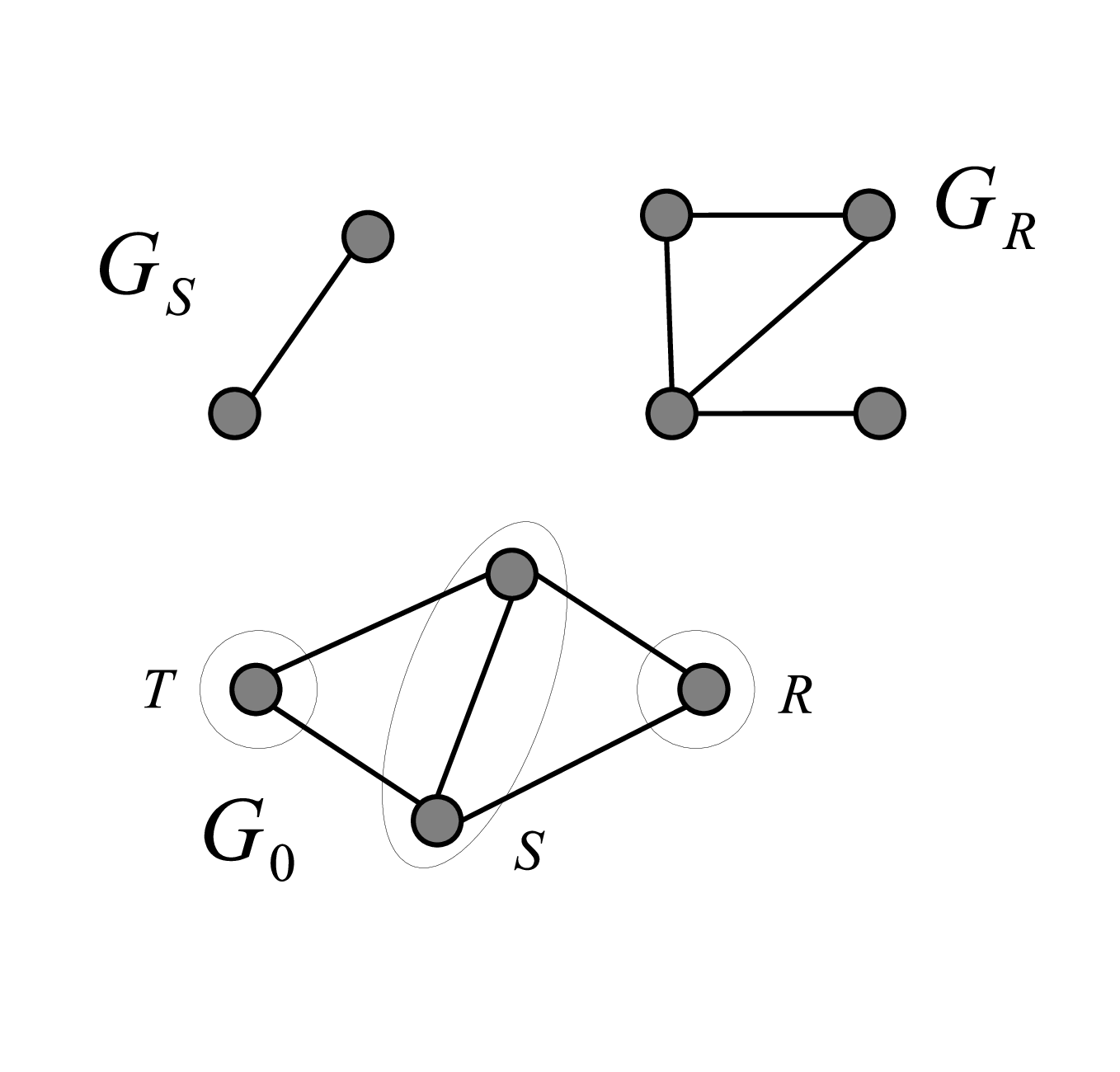}
\end{center}
\caption{The foundation and the towers that create $\tilde{G}$}
\label{fig:saturated_parts}
%\end{center}
\end{minipage}
\begin{minipage}{0.39\hsize}
\begin{center}
\includegraphics[width=5cm]{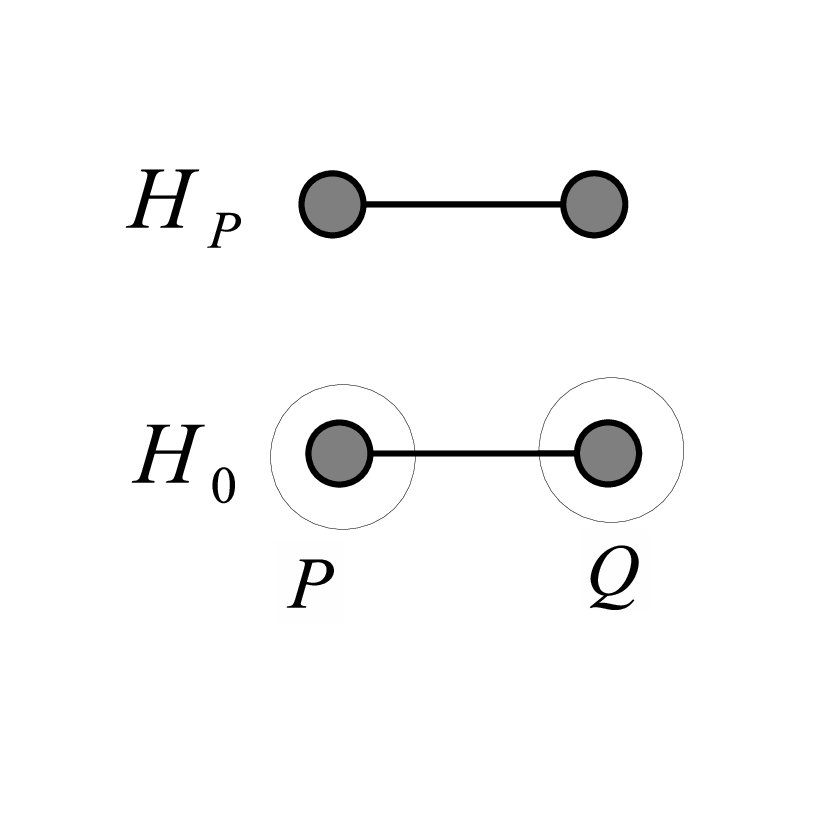}
\end{center}
\caption{The foundation and the towers that create $G_R$}
\label{fig:saturated_subparts}
%\end{center}
\end{minipage}
\end{figure}

\begin{figure}
\begin{minipage}{0.59\hsize}
\includegraphics[width=7cm]{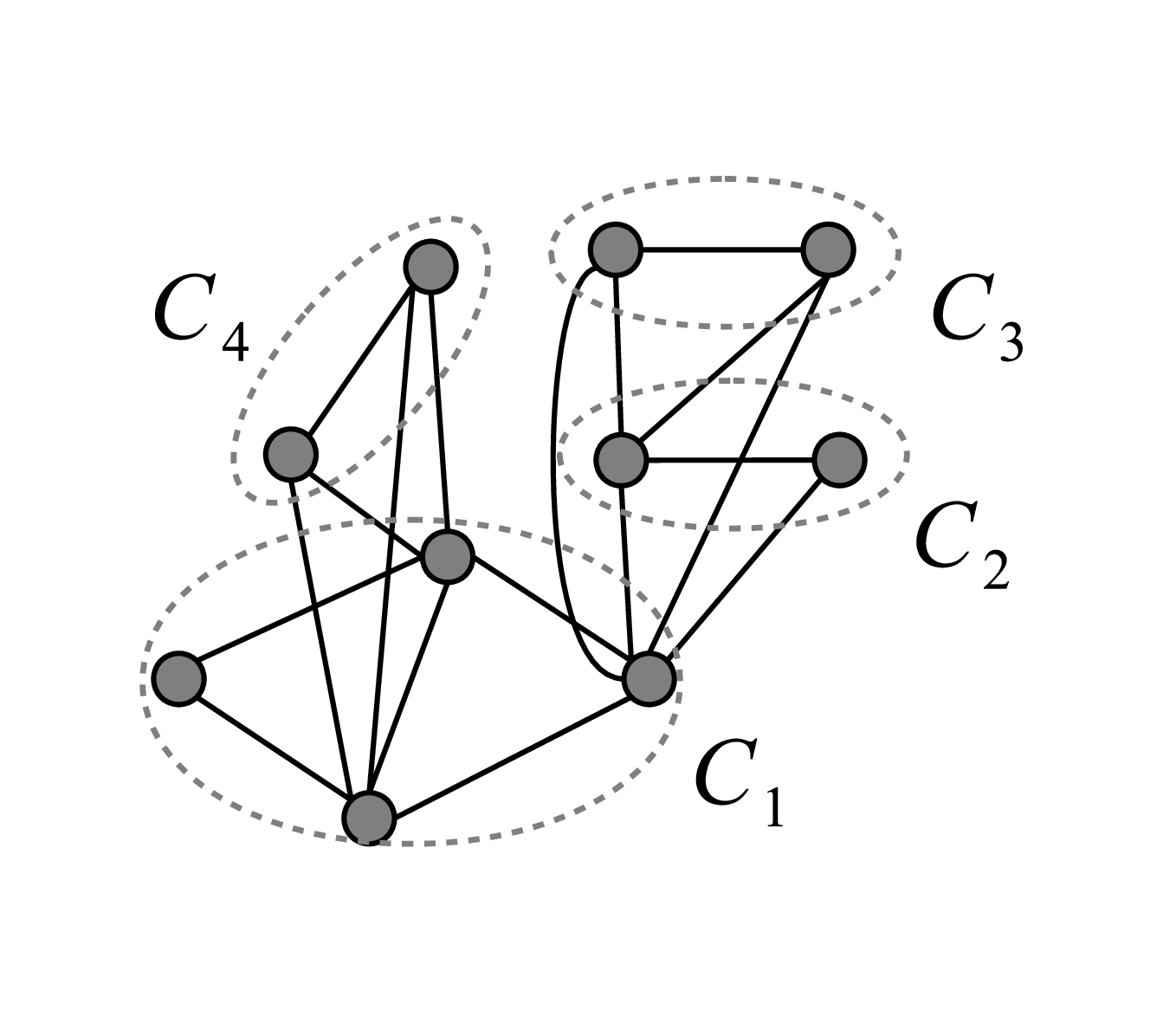}
\caption{The factor-connected components of $\tilde{G}$}
\label{fig:saturated_component}
\end{minipage}
%\begin{minipage}{0.49\hsize}
%\includegraphics[width=8cm]{../draw/saturated_partition.eps}
%\caption{The canonical partition of $\tilde{G}$}
%\label{fig:saturated_partition}
%\end{minipage}
\begin{minipage}{0.39\hsize}
\includegraphics[width=4cm]{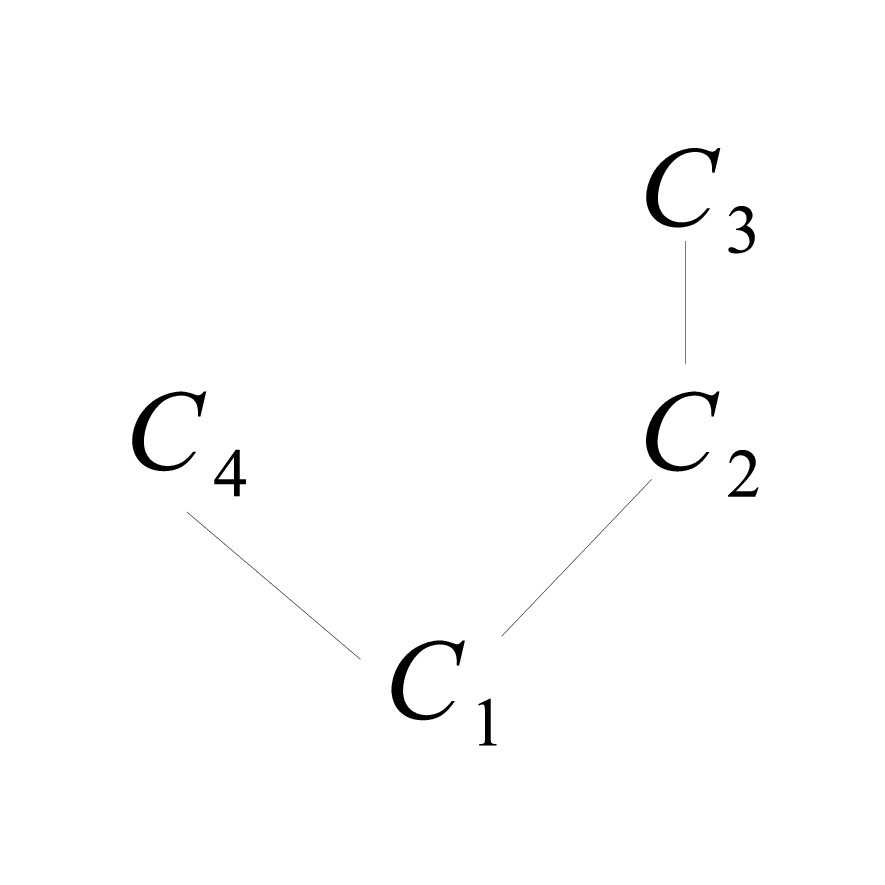}
\caption{The Hasse diagram of $(\mathcal{G}(\tilde{G}), \yield)$}
\label{fig:saturated_hasse}
\end{minipage}
\end{figure}

\begin{figure}
\begin{center}
\includegraphics[width=7cm]{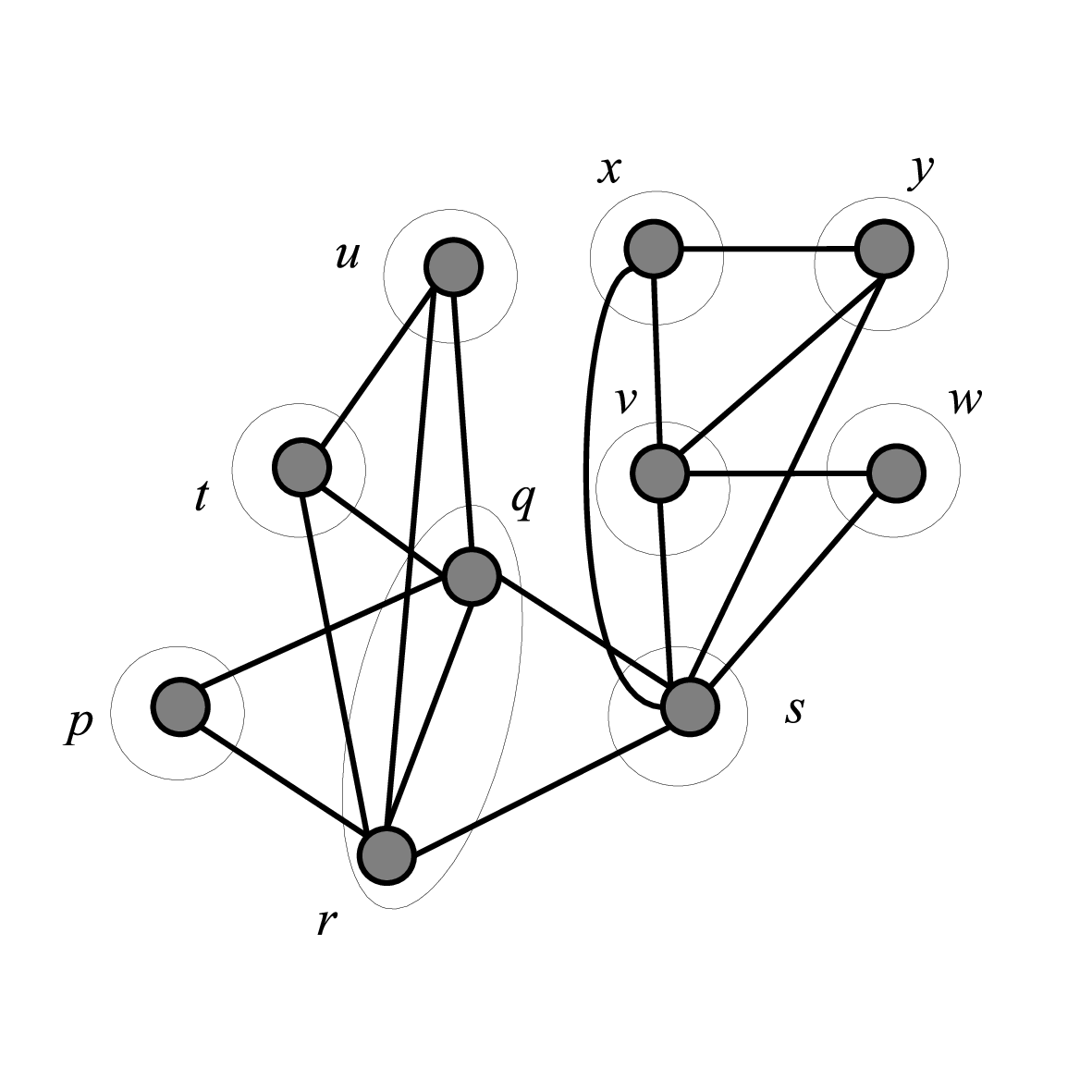}
\caption{The generalized canonical partition of $\tilde{G}$}
\label{fig:saturated_partition}
\end{center}
\end{figure}

\begin{comment}
\begin{figure}
\begin{center}
\includegraphics[width=15cm]{../draw/saturated_construct.eps}
\end{center}
\caption{A cathedral construction}
\label{fig:saturated_construct}
\begin{center}
\includegraphics[width=10cm]{../draw/saturated_subconstruct.eps}
\caption{The cathedral construction of $G_R$}
\label{fig:saturated_subconstruct} 
\end{center}
\end{figure}
\end{comment}
%
\begin{comment}
\begin{figure}
\begin{minipage}{0.49\hsize}
\begin{center}
\includegraphics[width=10cm]{../draw/saturated_parts.eps}
\caption{}
\end{center}
\end{minipage}
\begin{minipage}{0.49\hsize}
\begin{center}
\includegraphics[width=10cm]{../draw/saturatedgraph.eps}
\caption{}
\end{center}
\end{minipage}
\end{figure}
\end{comment}
%\begin{figure}
%\includegraphics[width=10cm]{../draw/saturated_substructure.eps}
%\caption{}
%\end{figure}

%
Figures~\ref{fig:saturated_verticesnumbered},  
\ref{fig:saturated_parts}, 
\ref{fig:saturated_subparts}
 show examples of the cathedral construction. 
In Figure~\ref{fig:saturated_parts}, 
the graph $G_0$ is an elementary saturated graph with 
the canonical partition $\gpart{G_0} = \{S, T, R\}$, 
and the graphs $G_S$, $G_T$, $G_R$ are saturated graphs 
such that $G_S$ and $G_R$ are respectively 
elementary and non-elementary while  
$G_T$ is an empty graph.    
If we conduct the cathedral construction 
with the foundation $G_0$ and the family of towers 
$\mathcal{T} = \{G_S, G_T, G_R\}$, 
we obtain the saturated graph $\tilde{G}$ in Figure~\ref{fig:saturated_verticesnumbered}. 
Moreover, 
Figure~\ref{fig:saturated_subparts} shows that 
if we conduct the cathedral construction with  
the foundation $H_0$ with $\gpart{H_0} = \{ P, Q\}$
 and the family of towers $\{H_P, H_Q\}$, where   
$H_P$ is an elementary saturated graph and  $H_Q$ is an empty graph, 
then we obtain the saturated graph $G_R$.  
(Therefore, in other words, the graph $\tilde{G}$ is constructed 
by a repetition of  the cathedral construction 
using the elementary saturated graphs $H_0$, $H_P$, $G_0$, and $G_S$ 
as fundamental building blocks.)

\begin{comment}
\begin{example}{\rm 
Figures~\ref{fig:saturated_construct} and \ref{fig:saturated_subconstruct} 
give examples of cathedral constructions. 
Figure~\ref{fig:saturated_construct} indicates the cathedral construction 
executed on the foundation $G_0$, 
which is an elementary saturated graph with 
the canonical partition $\gpart{G_0} = \{S, T, R\}$, 
and the towers 
$\{G_S, G_T, G_R\}$, 
where $G_T$ is an empty graph and each of $G_S, G_T, G_R$ is a saturated graph.  
The right graph in Figure~\ref{fig:saturated_construct} 
indicates the graph created by this cathedral construction, namely, 
by joining every vertex of $S$ and $G_S$, 
$T$ and $G_T$, and $R$ and $G_R$, respectively. 

Additionally, 
Figure~\ref{fig:saturated_subconstruct} shows that the saturated graph $G_R$ 
is created by executing the cathedral construction on 
the elementary saturated graph $H_0$ with the canonical partition 
$\gpart{H_0} = \{ P, Q\}$, and 
the towers $\{H_P, H_Q\}$, where $H_P$ is an elementary saturated graph 
and $H_Q$ is an empty graph. 
{\hfill $\blacksquare$}
}
\end{example}
\end{comment}

%
\begin{theorem}[The Cathedral Theorem~\cite{lovasz1972b, lp1986}]\label{thm:cathedral} 
A factorizable graph 
$G$ is saturated if and only if 
it is constructed from smaller saturated graphs by the cathedral construction. 
In other words, 
%That is to say, 
if a factorizable graph $G$ is saturated, 
then  
there is a subgraph $G_0$ and a family of subgraphs $\mathcal{T}$ of $G$ 
which are well-defined as a foundation and a family of towers,  
and $G$ is the graph constructed  from $G_0$ and $\mathcal{T}$ by the cathedral construction; 
conversely, 
if $G$ is a graph obtained from a foundation and towers by  the cathedral construction,  
then $G$ is saturated.   
% 
\begin{comment}
Let $G$ be a factorizable graph. 
%Let $G$ be a \matchablesp graph. 
Then, 
$G$ is saturated if and only if it is constructed by iterating  cathedral constructions;  
namely, $G$ is saturated if and only if 
there is a subgraph $G_0$ and a family of subgraphs $\mathcal{T}$ of $G$ 
which are well-defined as a foundation and towers,  
and $G$ is the graph constructed by conducting the cathedral construction to 
$G_0$ and $\mathcal{T}$.  
\end{comment}

Additionally, if $G$ is a saturated graph 
obtained  from  
a foundation  $G_0$ and a family of towers $\mathcal{T} = \{G_S\}_{S\in\gpart{G_0}}$ 
by the cathedral construction, then,
\begin{enumerate} 
\renewcommand{\labelenumi}{\theenumi}
\renewcommand{\labelenumi}{{\rm \theenumi}}
\renewcommand{\theenumi}{(\roman{enumi})}
\item \label{item:allow} %the allowed edges of $G$ are exactly ones of $G_0$ and $G_S$,
$e\in E(G)$ is allowed if and only if it is an allowed edge of $G_0$ or $G_S$ 
for some $S\in \gpart{G_0}$,  
\item \label{item:unique}
such $G_0$ uniquely exists;  
%\if0 $G_0$ is uniquely determined, \fi
 that is, if $G$ can be obtained  from a foundation $G'_0$ and  a family of towers $\mathcal{T}'$ by the cathedral construction, 
%\if0 a family of saturated graphs \fi
%$\mathcal{T}'$,   
then $V(G_0) = V(G'_0)$ holds, and  
%Consequently $\mathcal{T} = \mathcal{T}'$ also holds.  
\item \label{item:c_sat} $V(G_0)$ is exactly the set of vertices that 
is disjoint from $C(G-x)$ for any $x\in V(G)$.
\end{enumerate}
\end{theorem}
In the cathedral construction, 
each tower is saturated. 
%the graphs in the towers are each saturated.    
Therefore, 
the first sentence of Theorem~\ref{thm:cathedral} 
reveals a nested or inductive structure 
and 
gives 
a constructive characterization of the saturated graphs 
by the cathedral construction.  
%
\begin{comment}
The first sentence of Theorem~\ref{thm:cathedral} gives 
a constructive characterization of the saturated graphs 
by the cathedral construction;
\end{comment}
%
In this characterization, 
the elementary saturated graphs are the fundamental building blocks. 
Theorem~\ref{thm:cathedral} \ref{item:allow} tells that 
a set of edges in a saturated graph $G$ is a perfect matching 
if and only if it is 
a disjoint union of perfect matchings of 
the foundation and the towers that create $G$. 
Theorem~\ref{thm:cathedral} \ref{item:unique} tells that  
for each saturated graph, 
the way to construct it  uniquely exists, 
and \ref{item:c_sat} shows a relationship 
between the cathedral construction and the Gallai-Edmonds partition.

In the new proof, the following two theorems, Theorems~\ref{thm:cathedral_nec} and 
\ref{thm:cathedral_suff}, together with Theorem~\ref{thm:ge2cathedral}, 
will serve as nuclei,  referring to the special features of the poset and the canonical partition for saturated graphs. 
\begin{theorem}\label{thm:cathedral_nec}
If a factorizable graph $G$ is saturated,  
then the poset $(\mathcal{G}(G), \yield)$ has the minimum element, 
say $G_0$, 
and  it satisfies $\pargpart{G}{G_0} = \gpart{G_0}=: \mathcal{P}_0$. 
Additionally, for each $S\in \mathcal{P}_0$, 
the connected component $G_S$ of $G-V(G_0)$  
such that $\Gamma_{G}(G_S)\subseteq S$ 
exists uniquely or is an empty graph, and 
$G$ is the graph obtained from the foundation $G_0$ and the family of towers  
$\mathcal{T} := \{G_S\}_{S\in \mathcal{P}_0}$ by  the cathedral construction.   
\end{theorem}
\begin{theorem}\label{thm:cathedral_suff}
Let $G_0$ be a saturated elementary graph, 
and 
$\mathcal{T}:= \{G_S\}_{S\in\gpart{G_0}}$ be a family of saturated graphs. 
Let $G$ be the graph obtained 
from the foundation $G_0$ and the family of towers $\mathcal{T}$ 
by  the cathedral construction. 
Then, $G$ is saturated, 
$G_0$ forms a factor-connected component of $G$, 
that is, $G[V(G_0)]\in \mathcal{G}(G)$, 
and it is the minimum element of the poset $(\mathcal{G}(G), \yield )$. 
\end{theorem}

In the remaining part of this paper, 
we are going to prove Theorem~\ref{thm:cathedral_nec} and Theorem~\ref{thm:cathedral_suff} and then obtain Theorem~\ref{thm:cathedral}. 
With Theorem~\ref{thm:cathedral_nec} and Theorem~\ref{thm:cathedral_suff}, 
we obtain the constructive characterization of the saturated graphs.  We also obtain 
 a new characterization of foundations and families of towers, 
which gives a clear comprehension of saturated graphs 
by the canonical structures of factorizable graphs in Section~\ref{sec:canonical}. 
Thanks to this new characterization, 
the remaining statements of the cathedral theorem will be obtained quite smoothly. 

\subsection{Proof of Theorem~\ref{thm:cathedral_nec}}

Here we show some lemmas etc. to show that any saturated graph is constructed by the cathedral construction and prove Theorem~\ref{thm:cathedral_nec}.   
\begin{lemma}\label{lem:minimum}
If a \matchablesp graph $G$ is saturated,
then the poset $(\mathcal{G}(G), \yield)$ has the minimum element.
\end{lemma}
\begin{proof} 
Suppose the claim fails, that is, 
the poset has distinct minimal elements $G_1, G_2\in\mathcal{G}(G)$. 
Then, by Theorem~\ref{thm:add}, 
there exist possibly identical complement edges $e, f$ 
joining $V(G_1)$ and $V(G_2)$ 
such that  $\mathcal{G}(G+e+f) = \mathcal{G}(G)$. 
This means that adding $e$ or $f$ to $G$ 
does not create any new perfect matchings, 
which contradicts $G$ being saturated. 
%Obvious from Theorem~\ref{thm:add}.
%
\qed\end{proof}
In order to obtain Theorem~\ref{thm:cathedral_nec},  
by letting $G$ be a saturated graph, 
we show in the following that the minimum element 
$G_0$ of the poset by $\yield$ and the connected components of $G-V(G_0)$ are well-defined 
as a foundation and towers of the cathedral construction 
and $G$ is the graph obtained by the cathedral construction with them.

The next fact is easy to see from Fact~\ref{fact:deletable2path} 
and Property~\ref{prop:complement}.  
We will use this fact 
in the proofs of  Lemma~\ref{lem:gpartispart} and Lemma~\ref{lem:each_saturated} later. 
\begin{fact}\label{prop:complete}\label{lem:complete}\label{fact:complete}
%\begin{proposition}\label{prop:complete}
Let $G$ be a saturated graph, and let $H\in \mathcal{G}(G)$. 
Then, for any $u, v\in V(H)$ with $u\gsim{G} v$, 
$uv\in E(G)$. 
%\end{proposition}
\end{fact}
Next, we give the following lemma, which will contribute to the proofs of both of Theorems~\ref{thm:cathedral_nec} and \ref{thm:cathedral_suff}, actually. 
\begin{lemma}\label{lem:gpartispart}
Let $G$ be a saturated graph, 
and let $G_0\in\mathcal{G}(G)$. 
%$G_0$ be the minimum element of the poset $(\mathcal{G}(G), \yield)$. 
%such that 
%the poset $(\mathcal{G}(G), \yield)$ has the minimum element $G_0$. 
Then, $\pargpart{G}{G_0} = \gpart{G_0}$. 
\end{lemma}
\begin{proof}
Since we know by Fact~\ref{fact:refinement} 
that $\pargpart{G}{G_0}$ is a refinement of $\gpart{G_0}$, 
it suffices to prove 
that $\gpart{G_0}$ is a refinement of $\pargpart{G}{G_0}$,  
that is, 
if $u\gsim{G_0} v$, then $u\gsim{G} v$. 
We prove the contrapositive of this.  
 
%Since  we know that $\pargpart{G}{G_0}$ is a refinement of $\pargpart{G_0}{G_0}$,
%we should just prove if $u,v\in V(G_0)$ satisfy $u\not\gsim v$  in $G$,  
%then $u\not\gsim v$ in $G_0$.
%
Let $u,v\in V(G_0)$ with $u\not\gsim{G} v$. 
Let $M$ be a  perfect matching of $G$. 
By Fact~\ref{fact:deletable2path}, 
there are $M$-saturated paths between $u$ and $v$;    
let $P$ be a shortest one. 
Suppose $E(P)\setminus E(G_0) \neq \emptyset$, 
and let $Q$ be one of the connected components of $P - E(G_0)$,  
with end vertices $x$ and $y$. 
Since $Q$ is an $M$-ear relative to $G_0$ by Property~\ref{prop:separating2saturated}, 
$x\gsim{G} y$ follows by Proposition~\ref{prop:ear-base}.  
%Suppose $P-G_0 \neq \emptyset$,  and let  $Q$ be one of the  segments of $P-G_0$, whose 
%end vertices are $x$ and $y$.
%Then by Lemma~\ref{lem:distinctivebase},
%\if0 two end points $x$, $y$ of $Q$ \fi satisfy $x\gsim y$.
Therefore,  $xy\in E(G)$  by Fact~\ref{prop:complete}, which means 
we can get a shorter $M$-saturated path between $u$ and $v$ 
 by replacing $Q$ by $xy$ on $P$,
 a contradiction.   
Thus, we have $E(P)\setminus E(G_0) = \emptyset$;  
that is,  $P$ is a path of $G_0$. 
Accordingly, $u \not\gsim{G_0} v$ by Fact~\ref{fact:deletable2path}. 
% Thus, $P\subseteq G_0$, which means $u \gsim v $ in $G_0$ 
% by Proposition~\ref{prop:deletable2path}. 
%Take a $u$-$v$  $M$-saturated path $P$ for arbitrary perfect matching of $G$.
%There is an $M$-saturated path $P$ between $u$ and $v$.
%We proceed by induction on the number of segments of $P-V(G_0)$.
%If $P-V(G_0) = \emptyset$ we are done.
%Suppose not and let $xPy$  be one of the segments of $P-V(G_0)$ for $x, y\in V(G_0)$.
%%Then due to theorem~\ref{thm:equivalentdefinition} 
%All inner vertices of $P$ are in the same component of $G[\up{G_0}]$ and 
%by Lemma~ref{lem:base} $x\gsim y$. 
%Thus adding $G+xy$ never creates another matching,
%%and we can say $xy\in E(G)$.
%which means $xy\in E(G)$.
%So substituting $xPy$ by $xy$ gives another $u$-$v$ $M$-saturated path $P'$ 
%that has fewer segements of $P'-V(G_0)$ and by induction hypothesis we are done.
%
\qed\end{proof}
As we mention in Fact~\ref{fact:refinement}, 
for a factorizable graph $G$ and $H\in\mathcal{G}(G)$,  
$\pargpart{G}{H}$ is generally a refinement of $\gpart{H}$. 
However,  the above lemma  states that 
if $G$ is a saturated graph then they coincide. 
Therefore this lemma associates the generalized canonical partition 
with the cathedral theorem.

Next, note the following fact, which we present to prove Lemma~\ref{lem:each_saturated}:
\begin{fact}\label{fact:connected}
If a factorizable graph $G$ is saturated, 
$G$ is connected. 
\end{fact}
\begin{proof}
Suppose the claim fails, 
that is, $G$ has two distinct connected components, $K$ and $L$. 
Let $u\in V(K)$ and $v\in V(L)$, 
and let $M$ be a perfect matching of $G$. 
By Property~\ref{prop:complement},  
there is an $M$-saturated path between $u$ and $v$, 
contradicting the hypothesis that $K$ and $L$ are distinct. 
\qed\end{proof}
Before reading Lemma~\ref{lem:each_saturated}, 
note that if a factorizable graph $G$ has the minimum element $G_0$ 
for the poset $(\mathcal{G}(G), \yield)$, 
then for each connected component $K$ of $G-V(G_0)$, 
$\Gamma_{G}(K)\subseteq  V(G_0)$ holds. 
\begin{lemma}\label{lem:each_saturated}
Let $G$ be a saturated graph, 
and $G_0$ be the minimum element of the poset $(\mathcal{G}(G), \yield)$. 
%such that the poset 
%$(\mathcal{G}(G), \yield)$ has the minimum element $G_0$. 
Then, $G_0$ and the connected components of $G-V(G_0)$ are each saturated. 
Additionally, for each $S\in \pargpart{G}{G_0}$, 
a connected component $K$  of $G - V(G_0)$ such that 
$\Gamma_{G}(K) \subseteq S$ exists uniquely or does not exist.  
\end{lemma}
\begin{proof}
%Let $M$ be a perfect matching of $G$. 
We first prove that $G_0$ is saturated. 
Let $e = xy$ be a complement edge of $G_0$. 
By the contrapositive of Fact~\ref{prop:complete},  $x\not\gsim{G} y$, 
which means $x\not\gsim{G_0} y$ by Lemma~\ref{lem:gpartispart}. 
Therefore, 
by Fact~\ref{fact:deletable2path} and Property~\ref{prop:complement}, 
the complement edge $e$ creates a new perfect matching if it is added to $G_0$. 
Hence, $G_0$ is saturated.  

Now we move on to the remaining claims. 
Take $S\in \pargpart{G}{G_0}$ arbitrarily, 
and let $K_1, \ldots, K_l$ be the connected components of $G-V(G_0)$ 
which satisfy $\Gamma_{G}(K_i) \subseteq S$
 for each $i = 1,\ldots, l$. 
 Let $\hat{K} := G[ V(K_1)\dot{\cup}\cdots \dot{\cup} V(K_l)]$. 

We are going to obtain the remaining claims by showing that 
$\hat{K}$ is saturated.   
% We are going to show that $K$ is saturated. 
Now let $e = xy$ be a complement edge of $\hat{K}$,  i.e.,  $x,y\in V(\hat{K})$ 
and $xy\not\in E(\hat{K})$. 
Let $M$ be a perfect matching of $G$. 
With Property~\ref{prop:complement}, 
in order to show that $\what{K}$ is saturated 
it suffices to prove that there is an $M$-saturated path between 
$x$ and $y$ in $\what{K}$.  
Since $G$ is saturated,  there is an $M$-saturated path $P$ between $x$ and $y$ in $G$
by Property~\ref{prop:complement}. 
 
 Obviously by the definition, $\Gamma_{G}(\hat{K})\subseteq S$; 
on the other hand, $V(G)\setminus V(\hat{K})$ is of course a separating set. 
Therefore,  if $E(P) \setminus E(\hat{K}) \neq \emptyset$, 
each connected component of  $P - V(\hat{K})$ is an $M$-saturated path,  
both of whose end vertices are contained in $S$, 
by Property~\ref{prop:separating2saturated}.    
%Therefore,  if $E(P) \setminus E(\hat{K}) \neq \emptyset$, 
%each connected component of  $P - V(\hat{K})$ is an $M$-saturated path,  
%both of whose end vertices are contained in $S$, 
%by Proposition~\ref{prop:separating2saturated}, 
%since $V(G)\setminus V(\hat{K})$ is of course a separating set. 
This contradicts Fact~\ref{fact:deletable2path}. 
Hence, $E(P) \subseteq E(\hat{K})$, which means  
$\hat{K}$ is itself saturated. 
Thus, by Fact~\ref{fact:connected}, 
it follows $\hat{K}$ is connected, 
which is equivalent to $l = 1$. 
This completes the proof. 
\qed\end{proof}
By Lemma~\ref{lem:gpartispart} and Lemma~\ref{lem:each_saturated}, it follows that $G_0$ is well-defined as a foundation 
and the connected components of $G-V(G_0)$ are well-defined as towers 
(of course if indices out of $\gpart{G_0}$ are assigned to them appropriately). \begin{lemma}\label{lem:join}
Let $G$ be a saturated graph, 
and $G_0$ be the minimum element of the poset $(\mathcal{G}(G), \yield)$,   
%such that 
%the poset $(\mathcal{G}(G), \yield)$ has the minimum element $G_0$,
 and 
let $K$ be a connected component of $G-V(G_0)$, whose neighbors are in $S\in\pargpart{G}{G_0}$. 
Then, for any $u\in V(K)$ and for any $v\in S$,  $uv\in E(G)$. 
\end{lemma}
\begin{proof}
Suppose the claim fails, that is, 
there are $u\in V(K)$ and $v\in S$ such that $uv\not\in E(G)$.
Then, by Property~\ref{prop:complement}, 
there is an $M$-saturated path between $u$ and $v$, 
where $M$ is an arbitrary perfect matching of $G$. 
By the definitions, 
$V(K)\subseteq \vup{S}$; therefore, $u\in \vup{S}$. 
Hence, this contradicts \ref{item:nosaturate} of Lemma~\ref{lem:reach}, 
and we have the claim. 
%
%%%%%%%%%%%%%%% alternated %%%%%%%%%%%%%%%%%%%%%%%%%%%%%%%%%%%%%%%
\begin{comment}
Suppose the claim fails, that is, 
there are $u\in V(K)$ and $v\in S$ such that $uv\not\in E(G)$.
Then, since $G$ is saturated, for any perfect matching $M$ of $G$, 
there is an $M$-saturated path $P$ between $u$ and $v$
by Proposition~\ref{prop:saturate}.
Then, by Theorem~\ref{thm:base}, 
$P - V(K)$ contains an $M$-saturated path both of whose end vertices 
are contained in $S$, a contradicting Proposition~\ref{prop:deletable2path}. 
% 
\if0%%%%%%%%%%%%%%%%%%%%%%%%%%%%%%%%%%%%%%%%%%%%%
Then, since $G$ is saturated, 
 for any perfect matching $M$ of $G$,  there is an
$M$-alternating circuit $C$ containing $uv$ in $G + uv$. 
Then, by Theorem~\ref{thm:base}, 
$C -K$ contains an $M$-saturated path both of whose end vertices 
are contained in $S$, a contradiction. 
\fi%%%%%%%%%%%%%%%%%%%%%%%%%%%%%%%%%%%%%%%%%%%%
\end{comment}
%%%%%%%%%%%%%%% alternated close %%%%%%%%%%%%%%%%%%%%%%%%%%%%%%%%%%%%%%%%%%
%
\qed\end{proof}

Now we are ready to prove Theorem~\ref{thm:cathedral_nec}: 

\begin{proof}[Theorem~\ref{thm:cathedral_nec}]
The first sentence of Theorem~\ref{thm:cathedral_nec} is immediate from 
%The claims in the first sentence are immediate from 
Lemma~\ref{lem:minimum} and Lemma~\ref{lem:gpartispart}. 
The former of 
%The first half of 
the second sentence is also immediate by Lemma~\ref{lem:each_saturated}. 
%So is the first half of the second sentence by Lemma~\ref{lem:each_saturated}. 

%\begin{comment} 
For the remaining claim, 
first note that by Lemma~\ref{lem:each_saturated}, 
$G_0$ and any $G_S$ are saturated. 
Therefore, $G_0$ and $\mathcal{T} = \{G_S\}_{S\in\mathcal{P}_0}$ are well-defined 
as a foundation and a family of towers of the cathedral construction.

By the definition, 
for each $S\in\mathcal{P}_0$, 
it follows that $\Gamma_{G}(G_S) \subseteq S$. 
Additionally by Lemma~\ref{lem:join} 
every vertex of $V(G_S)$ and every vertex of $S$ are joined. 
Therefore, 
it follows that $G$ has a saturated subgraph $G'$
obtained from $G_0$ and $\mathcal{T}$ by the cathedral construction. 
Moreover, 
by Theorem~\ref{thm:base},  
for each connected component $K$ of $G-V(G_0)$ 
there exists $S\in\mathcal{P}_0$ such that $\Gamma(K)\subseteq S$; 
in other words, $K$ denotes the same subgraph of $G$ as $G_S$.    
Hence, 
$V(G) = V(G_0)\cup \bigcup_{S\in\mathcal{P}_0} V(G_S)$ holds 
and actually $G'$ is  $G$. 
Thus, $G$ is the graph obtained from $G_0$ and $\mathcal{T}$ by  the cathedral construction.  
%\end{comment}
%
\qed\end{proof} 

\subsection{Proof of Theorem~\ref{thm:cathedral_suff}}
Next we consider the graphs obtained by the cathedral construction 
and show Theorem~\ref{thm:cathedral_suff}, 
which states that the foundations of them are the minimum elements of the posets by $\yield$.  

Since the necessity of the first claim of Theorem~\ref{thm:cathedral}, the next proposition,  
is not so hard (see~\cite{lp1986}), 
we here present it without a proof. 
\begin{proposition}[Lov\'asz~\cite{lovasz1972b, lp1986}]\label{prop:cathedral_nec}
Let $G_0$ be a saturated elementary graph, 
and $\mathcal{T} = \{G_S\}_{S\in\gpart{G_0}}$ 
be a family of saturated graphs. 
Then, the graph $G$ obtained   
from the foundation $G_0$ and the family of towers $\mathcal{T}$
by the cathedral construction 
is saturated. 
\end{proposition}
We give one more lemma: 
\begin{lemma}\label{lem:fc}
Let $G$ be a saturated graph, obtained from the foundation \if0 an elementary saturated graph \fi $G_0$ and 
the family of towers 
\if0 a family of saturated graphs \fi 
$\{G_S\}_{S\in \gpart{G_0}}$
by  the cathedral construction.   
Then, $G':= G/V(G_0)$ is factor-critical. 
\end{lemma}
\begin{proof}
Let $M^S$ be a perfect matching of $G_S$ for each $S\in \gpart{G_0}$, 
and let $M := \bigcup_{S\in\gpart{G_0}} M^S$. 
Then, $M$ forms a near-perfect matching of $G'\if0 G/V(G_0) \fi$, exposing only 
the contracted vertex $g_0$ corresponding to $V(G_0)$. 
Take $u\in V(G')\setminus\{g_0\}$ arbitrarily 
and let $u'$ be the vertex such that $uu'\in M$. 
Since $uu'\in M\cap E(G')$ and $u'g_0\in E(G')\setminus M$, 
there is an $M$-balanced path from 
$u$ to $g_0$ in $G'$, namely,  the one with edges $\{uu', u'g_0\}$. 
%
%%%%%%%%%%%%%%%%%%%%%%%%%%%%%%%%%%%%%%%%%%%%%
\begin{comment}
For any $u\in V(G') \setminus\{g_0\}$, 
since $u'g_0\in V(G')$, 
there is an $M$-balanced path $uu' + u'g_0$ from 
$u$ to $g_0$. 
\end{comment}
%%%%%%%%%%%%%%%%%%%%%%%%%%%%%%%%%%%%%%%%%%%%%
Thus, by Property~\ref{prop:path2root},  
$G'$ is factor-critical. 
\qed\end{proof}

Now we shall prove Theorem~\ref{thm:cathedral_suff}: 
\begin{proof}[Theorem~\ref{thm:cathedral_suff}]
By Proposition~\ref{prop:cathedral_nec}, 
$G$ is saturated. 
Since we have Lemma~\ref{lem:fc},  
in order to complete the proof,  
it suffices to prove  $G_0 \in \mathcal{G}(G)$. 
Let $p$ be the number of non-empty graphs in $\mathcal{T}$. 
We proceed by induction on $p$. 
If $p=0$, the claim obviously follows. 
Let $p > 0$ and suppose the claim is true for $p-1$. 
Take a non-empty graph $G_S$ from $\mathcal{T} \if0 \{G_S\}_{S\in\gpart{G_0}} \fi$, 
and let $G' := G - V(G_S)$. 
Then, $G'$ is the graph obtained by  the cathedral construction 
with $G_0$ and 
$\mathcal{T} \setminus \{G_S\} \cup \{H_S\} $,  
\if0 $\{G_S\}_{S\in\gpart{G_0}}\setminus \{G_S\} \cup \{H_S\}$ \fi
where $H_S$ is an empty graph. 
Therefore, 
Proposition~\ref{prop:cathedral_nec} yields that 
$G'$ is saturated, 
and the induction hypothesis yields that 
$G_0\in\mathcal{G}(G')$
and $G_0$ is the minimum element of the poset $(\mathcal{G}(G'), \yield )$. 
%%%%%%%%%%%%%%%%%%%%%%%%%%%%%%%%%%%%
\begin{comment}
Since $G'$ is obtained by conducting cathedral construction 
to $G_0$ and 
$\mathcal{T} \setminus \{G_S\} \cup \{H_S\} $,  
\if0 $\{G_S\}_{S\in\gpart{G_0}}\setminus \{G_S\} \cup \{H_S\}$ \fi
where $H_S$ is an empty graph,  
by Proposition~\ref{prop:cathedral_nec}, 
$G'$ is saturated, and by the induction hypothesis, 
$G_0\in\mathcal{G}(G')$
and $G_0$ is the minimum element of the poset $(\mathcal{G}(G'), \yield )$. 
\end{comment}
%%%%%%%%%%%%%%%%%%%%%%%%%%%%%%%%%%%%% 
Thus, by Lemma~\ref{lem:gpartispart}, 
\begin{cclaim}\label{claim:reverse_refinement}
$\pargpart{G'}{G_0} = \gpart{G_0}$. 
\end{cclaim}
%Let $G_p$ is assigned to $S_p\in \pargpart{G_0}{G_0}$. 
%
Let $M'$ be a perfect matching of $G'$ and 
$M^S$ be a perfect matching of $G_S$, 
and construct a perfect matching $M:= M' \cup M^S$ of $G$.  
%Note that if we take a perfect matching $M'$ of $G'$ 
%and one $M_p$ of $G'$, 
%then $M:= M' \cup M_p$ is a perfect matching of $G$. 
\begin{cclaim}\label{claim:nonallowed}
No edge of $E_{G}[S, V(G_S)]$ is allowed in $G$. 
\end{cclaim}
\begin{proof}
Suppose the claim fails, that is, 
an edge $xy \in  E_{G}[S, V(G_S)]$ is allowed in $G$. 
Then, 
there is an $M$-saturated path $Q$ between $x$ and $y$ by Property~\ref{prop:allowed}, 
and 
$Q[V(G')]$ is an $M$-saturated path by Property~\ref{prop:separating2saturated}. 
Moreover, since $\Gamma_{G}(G_S)\cap V(G') \subseteq S$, 
it follows  that 
$Q[V(G')]$ is an $M$-saturated path of $G'$ 
between two vertices in $S$. 
With Fact~\ref{fact:deletable2path} 
this is a contradiction, because $S\in\pargpart{G'}{G_0}$ 
by Claim~\ref{claim:reverse_refinement}. 
Hence, we have the claim. 
\qed\end{proof}
By Claim~\ref{claim:nonallowed}, 
it follows that 
a set of edges is a perfect matching of $G$ 
if and only if it is a disjoint union of 
a perfect matching of $G'$ and $G_S$.
%any perfect matching of $G$ is a disjoint union of 
%a perfect matching of $G'$ and $G_S$. 
Thus, $G_0$ forms a factor-connected component of $G$, 
and we are done. 
\qed\end{proof}

\subsection{Proof of Theorem~\ref{thm:cathedral} and an Example} 

Now we can prove the cathedral theorem, 
combining Theorems~\ref{thm:cathedral_nec}, \ref{thm:cathedral_suff}, and 
\ref{thm:ge2cathedral}: 
\begin{proof}[Theorem~\ref{thm:cathedral}]
By Proposition~\ref{prop:cathedral_nec} 
and Theorem~\ref{thm:cathedral_nec}, 
the first claim of Theorem~\ref{thm:cathedral} is proved. 
The statement \ref{item:allow} is by Theorem~\ref{thm:cathedral_suff}, 
since it states that $G_0\in\mathcal{G}(G)$. 
The statement \ref{item:unique} is also by Theorem~\ref{thm:cathedral_suff}, 
since the poset $(\mathcal{G}(G), \yield)$ is a canonical notion. 
%\ref{item:allow} and \ref{item:unique} 
%follow by Theorem~\ref{thm:cathedral_suff}. 
The statement \ref{item:c_sat} is by combining Theorem~\ref{thm:cathedral_suff} 
and Theorem~\ref{thm:ge2cathedral}. 
%%%%%%%%%%%%%%%%%%%%%%%%%%%%%%%%
\begin{comment}
By Theorems~\ref{thm:ge2cathedral}, 
\ref{thm:cathedral_nec} and \ref{thm:cathedral_suff}, 
we obtain Theorem~\ref{thm:cathedral}. 
\end{comment}
%%%%%%%%%%%%%%%%%%%%%%%%%%%%%%%%
%
\qed\end{proof}

\begin{example}{\rm 
The graph $\tilde{G}$ in Figure~\ref{fig:saturated_verticesnumbered} 
consists of  four factor-connected components, say $C_1, \ldots, C_4$ 
in Figure~\ref{fig:saturated_component}, 
%Figure~\ref{fig:saturated_component} indicates 
%the factor-connected components of the saturated graph $\tilde{G}$ 
%given by Figure~\ref{fig:saturated_verticesnumbered}, 
and Figure~\ref{fig:saturated_hasse} 
shows the Hasse diagram of $(\mathcal{G}{(\tilde{G})}, \yield)$, 
which has the minimum element $C_1$,  
as stated in Lemma~\ref{lem:minimum}. 
Figure~\ref{fig:saturated_partition} indicates 
the generalized canonical partition of $\tilde{G}$: 
\begin{quote}
$\gpart{\tilde{G}} = \{ \{p\}, \{q,r\}, \{s\}, 
\{t\}, \{u\}, \{v\},  \{w\},  \{x\}, \{y\} \}$. 
\end{quote}
Here 
we have  $\pargpart{\tilde{G}}{C_i} = \gpart{C_i}$ 
for each $i = 1,\ldots, 4$, as stated in Lemma~\ref{lem:gpartispart}.  
From these two figures 
we see examples for other statements on 
the saturated graphs  in this section. 
{\hfill $\blacksquare$}
}
\end{example} 

\section{Concluding Remarks}\label{sec:conclusion}
Finally, we give some remarks. 

\begin{remark}{\rm 
Theorems~\ref{thm:cathedral_nec} and \ref{thm:cathedral_suff}
can be regarded as a refinement, 
and Theorem~\ref{thm:ge2cathedral} as a generalization
 of Theorem~\ref{thm:cathedral}, 
from the point of view 
of the canonical structures of Section~\ref{sec:canonical}. 
{\hfill $\blacksquare$} 
}
\end{remark}
 
\begin{remark}{\rm 
The poset $(\mathcal{G}(G), \yield)$ and $\gpart{G}$ 
can be computed in $O(|V(G)|\cdot |E(G)|)$ time~\cite{kita2012a, kita2012b}, 
where $G$ is any factorizable graph. 
Therefore,  
given a saturated graph, 
we can also find how it is constructed by iterating the cathedral construction 
in the above time by computing the associated poset and 
the generalized canonical partition. 
{\hfill $\blacksquare$}
} 
\end{remark}
\begin{remark}{\rm 
The canonical structures of general factorizable graphs 
 in Section~\ref{sec:canonical} 
can be obtained without 
the Gallai-Edmonds structure theorem nor the notion of barriers. 
The other properties we cite in this paper to 
prove the cathedral theorem are also obtained without them.  
Therefore, our proof shows that 
the cathedral theorem holds without assuming either of them. 
{\hfill $\blacksquare$} 
}
\end{remark} 

%\begin{remark}
With the whole proof, 
we can conclude that 
the structures in Section~\ref{sec:canonical} is 
what essentially underlie the cathedral theorem. 
We see how a factorizable graph leads to a saturated graph
having the same family of perfect matchings 
by sequentially adding complement edges. 
Our proof is quite a natural one  
because the cathedral theorem---a characterization of a class of graphs defined by a kind of edge-maximality ``saturated''---is derived as a consequence of considering edge-maximality over the underlying general structure. 
%a characterization of saturated graphs, 
%which is a class of graphs defined by a kind of edge-maximality---is derived as a consequence of considering edge-maximality over the underlying general structure. 
We hope yet more would be found on the field of  
counting the number of prefect matchings 
with the results in this paper and \cite{kita2012a, kita2012b, kita2012f, kita2013a}. 
%\end{remark}

\begin{notes}{\rm
The statements in \cite{kita2012b} or \cite{kita2013a} 
can be also found in \cite{kita2012a} or \cite{kita2012f}, respectively. 
The journal version of this paper will appear in 
Journal of the Operations Research Society of Japan. }
\end{notes}

\begin{acknowledgement}{\rm 
The author is grateful to  anonymous referees for giving various comments for logical forms of the paper as well as  many other suggestions about writing.  
The author also wishes to express her gratitude to 
Prof. Yoshiaki Oda for carefully reading the paper 
and giving lots of useful comments 
and to Prof. Hikoe Enomoto and Prof. Kenta Ozeki 
for many useful suggestions about writing. 
}
\end{acknowledgement}

\bibliographystyle{plain}

\setcounter{theorem}{0}
\renewcommand{\thetheorem}{A\arabic{theorem}}
\newpage
\section*{Appendix: Basic Properties on Matchings}
Here we present some basic properties about matchings.   
These are easy to observe and some of them might be regarded as folklores. 
\begin{property}\label{prop:path2root}
Let $M$ be a near-perfect matching of a graph $G$ that exposes $v\in \Vg$. 
Then, $G$ is factor-critical if and only if for any $u\in \Vg$ there exists 
an $M$-\zero path from $u$ to $v$.
\end{property}
\begin{proof}
Take $u\in V(G)$ arbitrarily. 
Since $G$ is factor-critical, there is a near-perfect matching $M'$ of $G$ 
exposing only $u$. 
Then, $G. M\Delta M'$ is an $M$-balanced path from $u$ to $v$, 
and the sufficiency part follows. 

Now suppose there is an $M$-balanced path $P$ from $u$ to $v$. 
Then, $M\Delta E(P)$ is a near-perfect matching of $G$ 
exposing $u$. 
Hence, the necessity part follows. 
\qed\end{proof}
\begin{property}\label{prop:allowed}
Let $G$ be a factorizable graph, $M$ be a perfect matching of $G$, 
and $e = xy \in E(G)$ be such that $e\not\in M$. 
The following three properties are equivalent: 
\renewcommand{\labelenumi}{\theenumi}
\renewcommand{\labelenumi}{{\rm \theenumi}}
\renewcommand{\theenumi}{(\roman{enumi})}
\begin{enumerate}
\item \label{item:allowed} The edge $e$ is allowed in $G$. 
\item \label{item:circuit} There is an $M$-alternating circuit $C$ such that $e\in E(C)$. 
\item \label{item:path} There is an $M$-saturated path between $x$ and $y$. 
\end{enumerate}
\end{property}
\begin{proof}
We first show that \ref{item:allowed} and \ref{item:circuit} 
are equivalent. 
Let $M'$ be a perfect matching of $G$ such that $e\in M'$. 
%For the sufficiency part, 
%take a perfect matching $N$ of $G$ such that $e\in N$. 
Then, $G. M\Delta M'$ has a connected component 
which is an $M$-alternating circuit containing $e$. 
Hence, \ref{item:allowed} yields \ref{item:circuit}. 

Now let $L:= M\Delta E(C)$. 
%For the necessity part, 
%let $N:= M\Delta E(C)$. 
Then, $L$ is a perfect matching of $G$ such that $e\in L$. 
Hence, \ref{item:circuit} yields \ref{item:allowed};  
consequently, they are equivalent. 

Since \ref{item:circuit} and \ref{item:path} are obviously equivalent, 
now we are done. 
\qed\end{proof}
\begin{property}\label{prop:deletable}
Let $G$ be a factorizable graph and $M$ be a perfect matching of $G$, 
and  let $u, v\in V(G)$. 
Then, 
$G - u - v$ is factorizable if and only if 
there is an $M$-saturated path of $G$ between $u$ and $v$. 
\end{property}
\begin{proof}
For the sufficiency part, 
let $M'$ be a perfect matching of $G-u-v$. 
Then, $G. M\Delta M'$ has a connected component 
which is an $M$-saturated path between $u$ and $v$. 
For the necessity part, 
let $P$ be an $M$-saturated path between $u$ and $v$. 
Then, $M\Delta E(P)$ is a perfect matching of $G - u - v$, 
and we are done. 
\qed\end{proof}
The next one follows easily from the definition of the separating sets (Definition~\ref{def:separating}). 
\begin{property}\label{prop:separating2saturated}
Let $G$ be a factorizable graph and $M$ be a perfect matching of $G$. 
Let $X\subseteq V(G)$ be a separating set  
and $P$ be an $M$-saturated path. 
Then, 
\renewcommand{\labelenumi}{\theenumi}
\renewcommand{\labelenumi}{{\rm \theenumi}}
\renewcommand{\theenumi}{(\roman{enumi})}
\begin{enumerate}
\item 
each connected component of $P[X]$ is an $M$-saturated path, and 
\item 
any connected component of $P - E(G[X])$ that 
does not contain any end vertices of $P$  is 
an $M$-ear relative to $X$. 
\end{enumerate}
\end{property}
The next property is immediate by Property~\ref{prop:allowed}  
and is used frequently in Section~\ref{sec:proof}. 
\begin{property}\label{prop:complement}
Let $G$ be a factorizable graph, $M$ be a perfect matching, 
and $x, y\in V(G)$ be such that  $xy\not\in E(G)$. 
Then, the following properties are equivalent: 
\begin{enumerate}
\renewcommand{\labelenumi}{\theenumi}
\renewcommand{\labelenumi}{{\rm \theenumi}}
\renewcommand{\theenumi}{(\roman{enumi})} 
\item The complement edge $xy$ creates a new perfect matching in $G + xy$. 
\item The edge $xy$ is allowed in $G + xy$. 
\item There is an $M$-saturated path between $x$ and $y$ in $G$. 
\end{enumerate}
\end{property}

\end{document}